\documentclass[12pt]{amsart}
\usepackage{amsmath,amscd,graphics,graphicx,color,a4wide,hyperref,verbatim}

\usepackage{multicol, fullpage}
\usepackage[usenames,dvipsnames]{xcolor} 
\usepackage{tikz, tikz-3dplot, pgfplots}
\usepackage{tkz-graph}
\usepackage{tikz-cd}
\usetikzlibrary[positioning,patterns] 
\usetikzlibrary{matrix,arrows,decorations.pathmorphing}

\usepackage{graphicx}
\usepackage{psfrag}
\usepackage{marvosym}
\usepackage{amsmath}
\usepackage{amsfonts}
\usepackage{amssymb}
\usepackage{amsthm}
\usepackage{mathrsfs}
\usepackage{enumitem}

\usepackage{ytableau}

\usepackage{comment}

\usepackage{calligra}
\usepackage{mathrsfs}
\DeclareMathAlphabet{\mathcalligra}{T1}{calligra}{m}{n}
\DeclareFontShape{T1}{calligra}{m}{n}{<->s*[1.5]callig15}{}

\newtheorem{theorem}{Theorem}[section]
\newtheorem{lemma}[theorem]{Lemma}
\newtheorem{lem-def}[theorem]{Lemma-definition}
\newtheorem{proposition}[theorem]{Proposition}
\newtheorem{prop-def}[theorem]{Proposition-definition}

\theoremstyle{definition}
\newtheorem{definition}[theorem]{Definition}

\newtheorem{example}[theorem]{Example}

\newtheorem{remark}[theorem]{Remark}

\newtheorem*{theoremstar}{Theorem}

\numberwithin{equation}{section}

\newtheorem{thm}{Theorem}[section] 
\theoremstyle{plain} 
\newcommand{\thistheoremname}{}
\newtheorem{genericthm}[thm]{\thistheoremname}
  
\newtheorem*{genericthm*}{\thistheoremname}
\newenvironment{namedthm*}[1]
  {\renewcommand{\thistheoremname}{#1}%
   \begin{genericthm*}}
  {\end{genericthm*}}

\newcommand{\CC} {\mathbb{C}}

\newcommand{\PP} {\mathbb{P}}

\newcommand{\ZZ} {\mathbb{Z}}

\newcommand {\shA} {\mathcal{A}}
\newcommand {\shB} {\mathcal{B}}
\newcommand {\shC} {\mathcal{C}}
\newcommand {\shD} {\mathcal{D}}

\newcommand {\shH} {\mathcal{H}}

\newcommand {\shJ} {\mathcal{J}}
\newcommand {\shK} {\mathcal{K}}

\newcommand {\shT} {\mathcal{T}}
\newcommand {\shP} {\mathcal{P}}

\newcommand {\sD} {\mathscr{D}}
\newcommand {\sE} {\mathscr{E}}

\newcommand {\sO} {\mathscr{O}}

\newcommand {\foa}  {\mathfrak{a}}

\newcommand {\cone} {\operatorname{cone}}

\newcommand {\eps} {\varepsilon}

\newcommand{\sExt}{\mathscr{E} \kern -3pt xt}

\newcommand {\Gr} {\operatorname{Gr}}

\newcommand {\Hom} {\operatorname{Hom}}
\newcommand {\sHom}{\mathscr{H}\kern-5pt\mathcalligra{om}}

\newcommand {\id} {\operatorname{id}}
\newcommand {\Id} {\operatorname{Id}}

\newcommand {\Ker} {\operatorname{Ker}}

\newcommand {\rank} {\operatorname{rank}}

\newcommand {\Spec} {\operatorname{Spec}}

\newcommand{\hpd}{{\natural}}
\newcommand{\ccap}{{\, \times^{}_{\PP(V^\vee)} \,}}
\newcommand{\ccup}{{\, \star \,}}
\newcommand{\rJ}{P}

\newcommand{\Perf}{\mathrm{Perf}}

\newcommand{\Bl}{\mathrm{Bl}}
\newcommand{\act}{\operatorname{act}}

\newcommand{\rBlA}{\widetilde{\shA}^{\rm ref}}
\newcommand{\join}{*}

\DeclareRobustCommand{\Sec}{\ifmmode\mathsection\else\textsection\fi}

\title[]{Categorical Duality between joins and intersections}
\author[Q.Y.\ JIANG, N.C.\ Leung]{Qingyuan Jiang, Naichung Conan Leung}

\address{Institute for Advanced Study, Einstein Drive, Princeton, NJ 08540, USA}\email{jiangqy@ias.edu}

\address{The Institute of Mathematical Sciences and Department of Mathematics,
The Chinese University of Hong Kong, Shatin, N.T., Hong Kong}\email{leung@math.cuhk.edu.hk}

\begin{document}

\begin{abstract}
\noindent Classically, the projective duality between joins of varieties and the intersections of varieties only holds in good cases. In this paper, we show that categorically, the duality between joins and intersections holds in the framework of homological projective duality (HPD) \cite{Kuz07HPD}, as long as the intersections have expected dimensions. 
This result together with its various applications provide further evidences for the proposal of homological projective geometry of Kuznetsov and Perry \cite{KP18}.
When the varieties are inside disjoint linear subspaces, our approach also provides a new proof of the main result ``formation of categorical joins commutes with HPD" of \cite{KP18}. We also introduce the concept of an $n$-HPD category, and study its properties and connections with joins and HPDs. 
\end{abstract}

\maketitle

\section{Introduction} \label{sec:intro} 
Linear duality is a natural reflexive correspondence between linear subspaces of a finite dimensional $k$-vector space $V$ (resp. projective space $\PP(V)$) and its dual vector space $V^\vee$ (resp. dual projective space $\PP(V^\vee)$). For a linear space $L \subset V$, its {\em  dual} or {\em orthogonal} linear subspace is defined to be $L^{\perp} : = \Ker\{V^\vee \to L^\vee\} \subset V^\vee$ (resp. $\PP(L)^\perp:=\PP(L^\perp) \subset \PP(V^\vee)$), and vice versa. Linear duality interchanges summations and intersections:
	\begin{align*}
	(L_1 + L_2)^\perp = L_1^\perp \cap L_2^\perp, \quad \text{and} \quad
\PP(L_1 + L_2)^\perp = \PP(L_1)^\perp \cap \PP(L_2)^\perp.
	\end{align*}

In projective geometry, linear duality can be remarkably extended to a reflexive correspondence, called {\em projective duality}, denoted by $X \mapsto X^\vee$, for all proper subvarieties of $\PP(V)$ and $\PP(V^\vee)$, see \cite{GKZ}.
The summation $ \PP(L_1 + L_2)$ becomes the {\em join} $X_1 \join X_2 \subset \PP(V)$ of two varieties $X_1,X_2 \subset \PP(V)$. It is natural to ask whether the following naive generalization
	$$(X_1 \join X_2)^\vee = X_1^\vee \cap X_2^\vee$$
holds. The answer is ``no" in general, see \S \ref{sec:intro:classical} for explanations and examples.

In this paper, we show above equality ``holds {\em categorically}". More precisely, in the homological framework, projective duality is replaced by another reflexive correspondence called {\em homological projective duality (HPD)}, introduced by Kuznetsov \cite{Kuz07HPD}, denoted by $X \mapsto X^\hpd$; and 
the join $X_1 \join X_2$ is replaced by {\em categorical join} $X_1 \ccup X_2$. Our main result is:
	\begin{theoremstar}[{\em HPD interchanges categorical joins and fiber products}] The following holds:
$$(X_1 \ccup X_2)^\hpd  \simeq X_1^\hpd \ccap X_2^\hpd,$$
provided that the intersections have expected dimensions. (See Main theorem in \S \ref{sec:intro:catjoin}.)
	\end{theoremstar}

\subsection{Classical projective geometry} \label{sec:intro:classical} For a projective variety $X$ with a morphism to $\PP(V) = \PP^n$, to study the linear section of $X$ it is natural to consider the {\em universal hyperplane (section) for the morphism $X \to \PP(V)$}:
	$$\shH_{X} : = \{ (x, [H]) \mid x \in H\} \subset X \times \PP(V^\vee),$$
where $H \subset \PP(V)$ is a hyperplane, which is in one-to-one correspondence with a point $[H] \in \PP(V^\vee)$. The fiber of projection $\shH_X \to \PP(V^\vee)$ over $[H]$ is just the hyperplane section $X_{H}: = X \cap H$ of $X$. The {\em projective dual of $X \to \PP(V)$} is defined to the critical values of  the family $\shH_X \to \PP(V)$, which can be equivalently described as:
	$$X^\vee : = \{[H] \in \PP(V^\vee) \mid X \cap H ~~\text{non-transversely} \} \subset \PP(V^\vee).$$
 For smooth projective varieties $X_1,X_2 \subset \PP(V) = \PP^n$, the {\em join of $X_1$ and $X_2$} is defined to be the closure of union of lines connecting points of $X_1$ and $X_2$, i.e.:
 	$$X_1 \join X_2 : = \overline{\bigcup_{\substack{x_1 \in X_1, x_2 \in X_2, x_1 \neq x_2}} \langle x_1,x_2 \rangle} \subset \PP(V),$$
where $\langle x_1, x_2 \rangle \simeq \PP^1$ denotes the unique line passing through $x_1$ and $x_2$.

It is also useful to consider the {\em abstract join $P(X_1,X_2)$ of $X_1$ and $X_2$}, defined by:
	\begin{align*}
	P(X_1,X_2) := \PP_{X_1 \times X_2}(\sO(-H_1) \oplus \sO(-H_2)),
	\end{align*}
where $H_k$ denotes the pull-back of hyperplane class of $\PP(V)$ to $X_k \to \PP(V)$, $k=1,2$. $P(X_1,X_2)$ parametrizes abstractly all the lines $\langle x_1, x_2\rangle$ formed by points $x_1 \in X_1$ and $x_2 \in X_2$, hence is a $\PP^1$-bundle over $X_1 \times X_2$, and comes with a natural evaluation map to $\PP(V\oplus V)$, sending the fiber $\PP^1_{(x_1, x_2)}$ to the corresponding line $\langle (x_1,0), (0,x_2)\rangle$ of $\PP(V\oplus V)$. The image of the map is denoted by $R(X_1,X_2)$, called the {\em ruled join} of $X_1$ and $X_2$. By construction, there is canonical chain of (rational) maps connecting the various joins:
	$$P(X_1,X_2) \to R(X_1,X_2) \dashrightarrow X_1 \join X_2$$
where the latter is a generically finite rational map. See \S \ref{sec:classical_joins} for more details.
 
The relation between joins and projective duality is partially reflected in the famous {\em Terracini Lemma} of projective geometry, which can be translated as:
		$$(X_1 \join X_2)^\vee \subset X_1^\vee \cap X_2^\vee, \footnote{ Also, the lemma states that the expected relation holds for ``general points of $X_1 \join X_2$". More precisely, for an open dense subset $U$, every point $z \in U$ with $z\in \langle x_1, x_2\rangle \subset X_1 \join X_2$, the tangent space of $X_1\join X_2$ at $z$ is just the linear span of tangent spaces $T_{x_1} X_1$ and $T_{x_2} X_2$. }$$ 

The equality $(X_1 \join X_2)^\vee = X_1^\vee \cap X_2^\vee$ holds in good cases, for example, if $X_1, X_2$ are linear subspaces, or if $X_1$ and $X_2$ lie in disjoint linear subspaces. However in general, the following issues may prevent the equality from happening:
\begin{enumerate}
	\item The behaviour of $X_1 \join X_2$ is ``uncontrolled" along the intersection $X_1 \cap X_2$;
	\item For a general point $x \in X_1 \join X_2$, there can be more than one lines $\langle x_1, x_2\rangle$ passing through $x$, where $x_1 \in X_1$, $x_2 \in X_2$ .
	\end{enumerate}

\begin{example} Let $X_1 = Q \subset \PP^2$ a quadric, $X_2  = \{p\} \subset \PP^2$ such that $p \notin Q$. Then $X_1 \join X_2 = \PP^2$, and for a general point $x \in X_1 \join X_2$, there are two lines passing through $x$. This corresponds to the fact that the map $R(X_1,X_2) \to X_1 \join X_2$ is a finite map of degree two, which is ramified along the two tangent lines $\ell_1, \ell_2$ of $Q$ passing through $p$. Then intersection of the duals $X_1^\vee \cap X_2^\vee = \{p\}^{\perp} \cap Q^\vee = \{[\ell_1], [\ell_2]\}$ are the two points corresponding to $\ell_1, \ell_2$. Then $(X_1 \join X_2)^\vee =( \PP^2)^\vee \neq X_1^\vee \cap X_2^\vee$. However $R(X_1,X_2)^\vee = X_1^\vee \cap X_2^\vee$ holds. 
\end{example}

The example shows that the equality may fail even when $X_1$ and $X_2$ are disjoint. However it also hints that the failure can in certain cases be corrected, if we resolve the second issue $(2)$ by taking finite morphism $R(X_1,X_2) \to X_1 \join X_2$ into considerations. 

\medskip
We will see in {\em homological} framework, for categorical join, the issue $(2)$ is resolved by a similar manner, and issue $(1)$ is resolved by ``refined blowing up" along intersection. Hence it is reasonable to expect that the duality between join and intersection may hold categorically.

\subsubsection*{Convention for the introduction} 
In the rest of introduction, we will state the constructions and results using {\em Kontsevich's convention} of noncommutative algebraic geometry (see, for example, \cite{KR}), and use $X, X_1,X_2, \ldots$ to denote {\em noncommutative varieties}, by which we mean admissible triangulated subcategories of the derived categories of projective varieties (with proper enhancements, $dg$-enhanced or stable $\infty$-enriched). For a commutative scheme $S$, a $S$-linear noncommutative variety $X \to S$ means a $S$-linear category; the fiber product $X_1 \times_S X_2$ of $X_1,X_2 \to S$ over $S$ means tensor product of categories over $S$, etc. The precise definitions will be reviewed in \S \ref{sec:bc}. The readers may regard them as honest varieties. In the main part of this paper, the definitions, results and proofs will be given for categories.

\subsection{Homological projective duality} 
The homological framework for projective geometry is set up by Kuznetsov \cite{Kuz07HPD} in his study of the theory of homological projective duality (HPD). The input data of HPD theory are {\em Lefschetz varieties} or {\em Lefschetz categories}. A Lefschetz variety (or category) is a smooth (possibly noncommutative) variety $X \to \PP(V)$ together with a semiorthogonal decomposition of the form:
	$$D(X) = \langle \shA_0, \shA_1(H), \ldots, \shA_{m-1}((m-1)H)\rangle,$$
where $\shA_0 \supset \cdots \shA_{m-1}$ is a chain of subcategories, $\shA_k(k H)$ denotes the image of $\shA_k$ under the autoequivalence $\otimes \sO(kH)$, $H$ is be hyperplane class of $\PP(V)$, see \S \ref{sec:lef}.

The output of HPD theory is the {\em HPD variety} (or more generally, {\em HPD category}) $X^\hpd \to \PP(V^\vee)$, which is defined as a $\PP(V^\vee)$-linear component of the derived category of universal hyperplane $\shH_X$, which captures the ``categorical change" of the family $D(\shH_X)$. The HPD variety $X^\hpd$ is a homological modification of the classical projective dual $X^\vee$. HPD is a reflexive correspondence $(X^\hpd)^\hpd \simeq X$. See \S \ref{sec:HPD} for more details.

\subsection{Categorical joins} \label{sec:intro:catjoin}In \cite{KP18}, Kuznetsov and Perry defines the categorical join $\shJ(X_1,X_2) \to \PP(V)$ for Lefschetz varieties $X_k \to \PP(V)$ which are supported in disjoint linear subspaces, i.e. $X_k \to \PP(V_k)$, $k = 1,2$, and $V = V_1 \oplus V_2$. 

In this paper we extend the definition to general Lefschetz varieties $X_1,X_2 \to \PP(V)$. In order to deal with intersection $X_1 \times_{\PP(V)} X_2$, we consider the blow up of the abstract join $P(X_1,X_2)$ along the intersection:
	$$ \widetilde{P}(X_1,X_2) := \Bl_{f^{-1}(X_1 \times_{\PP(V)} X_2)} P(X_1,X_2),$$
where $f$ is the map $P(X_1,X_2) \to R(X_1,X_2)$. Then $\widetilde{P}(X_1,X_2)$ is equipped with a map to $\PP(V)$. For simplicity from now on we will assume the fiber products $X_1 \times_{\PP(V)} X_2$ is of expected dimension, otherwise we will need to consider {\em derived} fiber product instead of ordinary one.

 The {\em categorical join $X_1\ccup X_2$ of $X_1$ and $X_2$} is a $\PP(V)$-linear subcategory of $\widetilde{P}(X_1,X_2)$ defined by the following three steps (see \S \ref{sec:cat_joins} and Def. \ref{defn:cat_joins}):

\begin{enumerate}
	\item[Step 1.] Following \cite{KP18}, we can define first the {\em categorical ruled join} $\shJ(X_1,X_2)$ as a $\PP(V\oplus V)$-linear subcategory of $D(P(X_1,X_2))$, which is a homological modification of the ruled join $R(X_1,X_2) \subset \PP(V \oplus V)$.
	\item[Step 2.] Then we blow-up the categorical ruled join $\shJ(X_1,X_2)$ along the intersection, hence obtain a $\PP(V)$-linear category $\widetilde{\shJ}(X_1,X_2) \subset D(\widetilde{P}(X_1,X_2))$.
	\item[Step 3.] Finally we remove canonically the redundant components of $\widetilde{\shJ}(X_1,X_2)$ and obtain a $\PP(V)$-linear category $X_1 \ccup X_2$, which is defined to be the {\em categorical join}.
\end{enumerate}

We show that the categorical join $X_1 \ccup X_2$ is a $\PP(V)$-linear (moderate) Lefschetz category, with Lefschetz components explicitly given by Prop. \ref{prop:lef:cat_join}. In the case when $X_1$ and $X_2$ are disjoint $X_1 \times_{\PP(V)} X_2= \emptyset$, this definition agrees with the one given in \cite{KP18}. 

Our main result is the duality of categorical join and fiber products:

\begin{namedthm*}{Main theorem}[See Thm. \ref{thm:main}] \label{thm:intro}  Let $X_1,X_2 \to \PP(V)$ be Lefschetz varieties such that $X_1 \times_{\PP(V)} X_2$ is of expected dimension. Then there is a natural equivalence \footnote{Note in this introduction we are using Kontsevich's noncommutative geometry convention. R.H.S is the fiber product of noncommutative varieties. If $X_k^\hpd$'s are commutative varieties, then R.H.S. is equal to the derived category $D(X_1^\hpd \ccap X_2^\hpd)$ if the (underived) fiber product $X_1^\hpd \ccap X_2^\hpd$ is of expected dimension.}:
$$ (X_1 \ccup X_2)^\hpd  \simeq X_1^\hpd \ccap X_2^\hpd. 
$$
\end{namedthm*}

In the case when the intersection $X_1 \times_{\PP(V)} X_2 \neq \emptyset$, the formulation of the main theorem itself has direct nontrivial consequences.

\begin{example}[see Ex. \ref{eg:Gr(2,5)} \& \ref{eg:Gr-dual}] Let $X_1 = \Gr(2,5)\subset \PP^9$ (by Pl\"ucker embedding), and $X_2 = g \cdot \Gr(2,5)$ where $g$ is a generic element of ${\rm PGL}(10,\CC)$. Then the categorical joins is
	$$X_1 \ccup X_2 = \Gr(2,5) \ccup (g \cdot \Gr(2,5)) = \big \langle \sE, \sE(H), \ldots, \sE(8H)\big \rangle,$$
	where $\sE = D(X_1\cap X_2)$ is the derived category of the Calabi-Yau threefold $X_1 \cap X_2 =\Gr(2,5) \cap g \cdot \Gr(2,5)$. The theorem implies $X_1 \ccup X_2$ is HPD to the intersection $X_1^\vee \cap X_2^\vee$, which is another Calabi-Yau threefold that is {\em not birationally equivalent} to $X_1 \cap X_2$ (see \cite{OR, BCP}). Therefore our main theorem implies, and provides another proof (cf. \cite{JLX17, KP18}) of the fact that $D(X_1 \cap X_2) \simeq D(X_1^\vee \cap X_2^\vee)$.
\end{example}

In the disjoint case $X_1 \times_{\PP(V)} X_2= \emptyset$, the categorical joins and intersections of HPD admit concrete descriptions (which are related by HPD by our main theorem), see \S \ref{sec:disjoint-sit}, Thm. \ref{thm:min-rep} and Thm. \ref{cor:n-HPD}. These results hold for general $n$.

The splitting case is an important situation when $X_1 \times_{\PP(V)} X_2= \emptyset$ holds, where one assumes that $X_k \to \PP(V_k) \subset \PP(V)$, $k = 1,2$, and $V = V_1 \oplus V_2$. Then there are two possible HPDs of $X_k$: one is the HPD over the ambient space $\PP(V)$, denoted by $ (X_{k})_{/\PP(V)}^{\hpd}$; one is the HPD over $\PP(V_k)$, denoted by $ (X_{k})_{/\PP(V_k)}^{\hpd}$. We show that they are related by:

\begin{theoremstar}(See Thm. \ref{thm:join-linear}, \ref{thm:cone}) If $X_k \to \PP(V_k) \subset \PP(V)$ are Lefschetz varieties, $k=1,2$ and $V = V_1 \oplus V_2$. There are natural equivalences of categories:
\begin{align*}
	&(1) \quad (X_{1})_{/\PP(V)}^{\hpd} \simeq  (X_{1})_{/\PP(V_1)}^{\hpd} \ccup \PP(V_2^\vee) \quad\text{and}\quad  (X_{2})_{/\PP(V)}^{\hpd} \simeq  (X_{2})_{/\PP(V_2)}^{\hpd} \ccup \PP(V_1^\vee);\\
	&(2) \quad  (X_{1})_{/\PP(V)}^{\hpd} \ccap  (X_{2})_{/\PP(V)}^{\hpd}  \simeq  (X_{1})_{/\PP(V_1)}^{\hpd}\ccup  (X_{1})_{/\PP(V_2)}^{\hpd};\\
	&(3) \quad  (X_1 \ccup X_2)_{/\PP(V)}^{\hpd} \simeq  (X_{1})_{/\PP(V_1)}^{\hpd} \ccup  (X_{2})_{/\PP(V_2)}^{\hpd}.
	\end{align*}
\end{theoremstar}

The statement $(3)$ of above theorem is the main result ``formation of categorical joins commutes with HPD" of \cite{KP18}. Note that statement $(3)$ is the equivalent form of our Main theorem in the splitting case, and our approach in this paper provides a different proof of it.

\subsection{$n$-HPD category} One key ingredient of our approach is the concept of {\em $n$-HPD category}, which naturally relates the HPD of joins and the fiber products of HPDs (and also the join of HPDs). The definition of $n$-HPD category generalises that of a HPD category for $X \to \PP(V)$ to $n$ Lefschetz varieties $X_k \to \PP(V)$, $k=1,\ldots,n$, where ordinary HPD category corresponds to the case $n=1$. We illustrate by the case $n=2$.

The {\em double universal hyperplane} $\shH(X_1,X_2)$ for $X_1,X_2 \to \PP(V)$ is defined by: 
	$$\shH(X_1,X_2) = \{(x_1, x_2, [H])  \mid x_1 \in H, x_2 \in H\} \subset X_1 \times X_2 \times \PP(V^\vee),$$
which, as a family over $\PP(V^\vee)$, captures the simultaneous hyperplane sections of all $X_k$'s. The {\em double HPD category}  $\shC$ is a $\PP(V^\vee)$-linear subcategory of the derived category of $\shH(X_1,X_2)$, which captures the ``deepest strata of categorical changes" of the family $\shH(X_1,X_2)$ over $\PP(V^\vee)$ (see Lem. \ref{lem:sod:nH} for the precise meaning). The double HPD category $\shC$ can also be intrinsically defined in a similar manner as how the ordinary HPD is defined, i.e.
$$\shC  =  \{C \in D(\shH(X_1,X_2))  \mid \delta_{{\shH}*} \,C \in \shA^{(1)}_0 \boxtimes \shA^{(2)}_0 \boxtimes D(\PP(V^\vee))\} \subset D(\shH(X_1,X_2)),$$
where $\delta_{\shH} \colon \shH(X_1,X_2) \to X_1 \times X_2 \times \PP(V^\vee)$ is the inclusion, see Def. \ref{def:nHPD}.

\subsection{Strategy of proof} Our strategy, following the general strategy of {\em ``chess game"} of \cite{JLX17}, is to put all the categories of interests ($\sD_1, \sD_2, \sD_3, \shC$ below) inside the same ambient category $D(\shH(X_1,X_2))$ and naturally compare them. (Note that, however, in this paper we do {\em not} use the general results on ``chess game" of \cite{JLX17} except from following the philosophy). There are three  geometric pictures which connect the categories of interests to  $D(\shH(X_1,X_2))$ and hence with the double HPD category $\shC$:

The first picture is that there exists a natural {\em birational} morphism from the generalized universal hypersurface $\shH_{P(X_1,X_2)}$ of the abstract join $P(X_1,X_2)$ to the product space $X_1 \times X_2 \times \PP(V)$, which is an isomorphism exactly outside the double universal hyperplane $\shH(X_1,X_2) \subset  X_1 \times X_2 \times \PP(V)$, and it is a $\PP^1$-bundle along $\shH(X_1,X_2)$:
\begin{equation*}
	\begin{tikzcd}
	\PP(N_{\delta}^\vee) \ar{d}[swap]{\PP^1\text{-bundle}} \ar[hook]{r}{j} & \shH_{P(X_1,X_2)}  \ar{d}{\text{birational}}  \\
	\shH(X_1,X_2) \ar[hook]{r}{\delta}         &  X_1 \times X_2 \times \PP(V^\vee). 
	\end{tikzcd}	
\end{equation*}
This picture is the key to relate the HPD $\sD_1 :=(X_1 \ccup X_2)^\hpd \subset D(\shH_{P(X_1,X_2)})$ of the join $X_1 \ccup X_2$ with the double HPD category $\shC \subset D(\shH(X_1,X_2))$, and to show $\sD_1 \simeq \shC$.

The second picture is that $\shH(X_1,X_2) $ is the fiber product of $\shH_{X_1}$ and $\shH_{X_2}$ over $\PP(V^\vee)$. This geometry enables us to show $\shC \simeq X_1^\hpd \ccap X_2^\hpd =: \sD_2$. 

The third picture occurs in the splitting case $X_1 \to \PP(V_1)$, $X_2 \to \PP(V_2)$ and $V = V_1 \oplus V_2$. Then the join $\check{P}(\shH_{X_1}, \shH_{X_2})$ of universal hyperplanes $\shH_{X_1}$ and $\shH_{X_2}$ over $\PP(V^\vee)$ (where $\shH_{X_k}$ is the small universal hyperplane of $X_k$ over $\PP(V_k)$) is the {\em blowing up} of $\shH(X_1,X_2)$ along the subvariety $\shH_{X_1} \times X_2  \sqcup X_1 \times \shH_{X_2} $, 

\begin{equation*}
	\begin{tikzcd}
	E_1 \sqcup E_2 \ar{d}[swap]{\text{proj. bundle}} \ar[hook]{r}  	&  \ar{d}{\text{blow-up}} \check{P}(\shH_{X_1}, \shH_{X_2})
	\\
	\shH_{X_1} \times X_2  \sqcup X_1 \times \shH_{X_2}  \ar[hook]{r}         & \shH(X_1,X_2),
	\end{tikzcd}	
\end{equation*}
where the exceptional divisor satisfies $E_1 \simeq E_2 \simeq \shH_{X_1} \times \shH_{X_2}$. 
This picture enables us to directly compare $\shC$ with $\sD_3 := (X_1)_{/\PP(V_1)}^\hpd  \ccup (X_2)_{/\PP(V_2)}^\hpd \subset D(\check{P}(\shH_{X_1}, \shH_{X_2}))$.

The details of the proofs are given subsequently in section \ref{sec:main-results}.

\subsection{Homological projective geometry} In \cite{KP18} Kuznetsov and Perry proposed a robust theory called {\em homological projective geometry},  where the category $\underline{{\rm Lef}}_{/\PP(V)}$ of (smooth proper) $\PP(V)$-linear Lefschetz categories plays the role of the category of (smooth) projective subvarieties of $\PP(V)$ in classical projective geometry. More over, HPD, categorical joins and cones play the roles of projective duality,  joins and cones in projective geometry.

The known results supporting the proposal of homological projective geometry have been very fruitful and powerful. Kuznetsov's {\em fundamental theorem of HPD} \cite{Kuz07HPD}, as a homological counterpart of classical Lefschetz theory, systematically compares linear sections of HPDs.
The {\em categorical Pl\"ucker formula} \cite{JLX17}, as a two-step categorification of the topological Pl\"ucker formula, systematically compares the intersection of Lefschetz varieties and of their HPDs. Another formulation called {\em nonlinear HPD theorem} has been given in \cite{KP18}, where the authors introduce {\em categorical joins} for varieties inside different projective spaces.

This paper provides further strong evidences for this proposal. As a consequence of our main theorem, we show that the category $\underline{{\rm Lef}}_{/\PP(V)}$ of Lefschetz categories is closed under the following two {\em commutative} and {\em associative monoidal operations}, namely categorical join
	$$  \ccup   \colon \underline{{\rm Lef}}_{/\PP(V)} \times \underline{{\rm Lef}}_{/\PP(V)} \to \underline{{\rm Lef}}_{/\PP(V)}$$
 and fiber product
 	$$  \times_{\PP(V)}   \colon \underline{{\rm Lef}}_{/\PP(V)} \times \underline{{\rm Lef}}_{/\PP(V)} \to \underline{{\rm Lef}}_{/\PP(V)},$$
(if we assume the fiber products are smooth of expected dimensions), and that these two operations are {\em dual} to each other under HPD.  See Thm. \ref{thm:join-nHPD}.

For a linear subbundle $L \subset V^\vee$, we introduced in \cite{JL18Bl} the operation of  refined blow-up:
	 $$ \Phi  \colon \underline{{\rm Lef}}_{/\PP(V)}  \to \underline{{\rm Lef}}_{/\PP(L^\vee)}, \qquad \shA \mapsto \Phi(\shA) :=\Bl_{\shA_{\PP(L^\perp)}}^{\rm ref} \shA,$$
and show that it is dual to the operation of restrictions to linear subspaces. More precisely, $\Phi(\shA)^\hpd = (\shA^\hpd)|_{\PP(L)}$, for $\shA \in  \underline{{\rm Lef}}_{/\PP(V^)}$.
There is also an operation in the other direction, called {\em categorical cone},
	$\shC_{\PP(L^\vee)} (-)\colon \underline{{\rm Lef}}_{/\PP(L^\vee)}  \to \underline{{\rm Lef}}_{/\PP(V)},$
which is the categorical join $(-) \ccup \PP(L^\perp)$ if there is a splitting $V = L^\vee \oplus L^\perp$, see \cite{KP18}, also Rmk. \ref{rmk:cone}.
Then the results of this paper can be used to show that the composition:
		$$\Phi  \circ \shC_{\PP(L^\vee)} (-) \colon   \underline{{\rm Lef}}_{/\PP(L^\vee)}  \to \underline{{\rm Lef}}_{/\PP(V)} \to \underline{{\rm Lef}}_{/\PP(L^\vee)} $$
is equivalent to identity, i.e. $\Phi (\shC_{\PP(L^\vee)} (\shA)) \simeq \shA$, for $\shA \in \underline{{\rm Lef}}_{/\PP(L^\vee)}$.

\subsection{Conventions} Let $B$ be a fixed base scheme, smooth over a ground field of characteristic zero, and $V$, $V^\vee$ be dual vector bundles of rank $N$ over $B$. All schemes considered in this paper will be $B$-schemes, and products are fiber products over $B$. For a scheme $X$, the categories $\Perf(X)$, $D(X)$ and $D_{qc}(X)$ denote respectively the triangulated category of perfect complexes, the bounded derived category of coherent sheaves and the unbounded derived category of quasi-coherent sheaves on $X$. We will mainly focus on $D(X)$, however most of our results directly work for $\Perf(X)$.  A $B$-linear category will be an {\em admissible} subcategory $\shA \subset D(X)$ for some $B$-scheme $X$. Functors considered in this paper are all {\em derived} unless otherwise specified. The notation $\Phi \colon \shA  \rightleftarrows \shB \colon \Psi$ means that the functor $\Phi \colon \shA \to \shB$ is left adjoint to the functor $\Psi \colon \shB \to \shA$.

In this paper we will follow the philosophy of Bondal, Orlov, Kuznetsov, etc and use their well-established frameworks of {\em admissible} subcategories of (smooth) projective varieties. The readers should have no difficulty in translating the constructions and arguments to noncommutative algebraic geometry setting of using stable $\infty$-categories or dg-categories.

\subsection*{Acknowledgement} The authors would like to thank Richard Thomas, Mikhail Kapranov,  Andrei C\v{a}ld\v{a}raru, Matthew Young, Ying Xie and Dan Wang for many helpful discussions. J.Q.-Y. is supported by Grant from National Science Foundation (Grant No. DMS -- 1638352) and the Shiing-Shen Chern Membership Fund of IAS; L.N.C. is supported by grant from the Research Grants Council of the Hong Kong Special Administrative Region, China (Project No. CUHK -- 14301117).

\section{Preliminaries}

\subsection{Generalities} The readers are referred to \cite{Huy,BO,Kuz14SOD} for basic definitions of derived categories, and properties of semiorthogonal decompositions. A full triangulated subcategory $\shA$ of a triangulated category $\shT$ is called {\em admissible} if the inclusion functor $\gamma: \shA \hookrightarrow \shT$ has both a right adjoint functor $\gamma^!: \shT \to \shA$ and a left adjoint functor $\gamma^*: \shT \to \shA$. If $\shA\subset \shT$ is admissible, then its {\em right orthogonal} and respectively {\em left orthgonal}:
	$$\shA^\perp = \{ T \in \shT \mid \Hom(\shA,T) = 0\}
	\quad \text{and} \quad 
	{}^\perp \shA =\{ T \in \shT \mid \Hom(T, \shA) = 0\}$$
are both admissible. A {\em semiorthogonal decomposition} of a triangulated category $\shT$ (sometimes simply called a decomposition of $\shT$), denoted by
	$$\shT = \langle \shA_1, \ldots, \shA_n \rangle,$$
is a sequence of admissible full triangulated subcategories $\shA_1, \shA_2, \ldots, \shA_{n}$, such that $(1)$ $\Hom (a_j ,a_i) = 0$ for all $a_i \in \shA_i$ and $a_j \in \shA_j$ , if $j > i$, and $(2)$ the sequence generate the whole $D(X)$. Note that for an admissible subcategory $\shA \subset \shT$, we have semiorthogonal decompositions $\shT = \langle \shA^\perp, \shA \rangle =  \langle \shA, {}^\perp \shA\rangle$. 

\subsection{Derived categories over a base and base-change} \label{sec:bc} The references for this section are \cite{Kuz07HPD,Kuz11Bas}, see also summaries in \cite{JL18Bl}. Let $S$ be a fixed scheme, and $a \colon X \to S$ be a $S$-scheme. Then $D(X)$ is naturally equipped with {\em $S$-linear structure} given by $A \otimes a^* F$, for any $F \in \Perf(S)$ and $A \in D(X)$. An admissible subcategory $\shA \subset D(X)$ is called {\em $S$-linear} if $A \otimes a^* F \in \shA$ for all $A \in \shA$ and $F \in \Perf(S)$. Such an admissible subcategory $\shA$ will be simply referred to as an {\em $S$-linear category}. An {\em $S$-linear} functor between $S$-linear categories is an exact functor functorially preserving $S$-linear structures. Note that an $S$-linear category $\shA$ is by definition equipped with an action functor:
	$$\act \colon \shA \boxtimes D(S) \to \shA,$$
given by $\act(A \boxtimes F) = A \otimes a^* F \in \shA$. (The notation $\shA \boxtimes D(S)$ is to be defined later).

A base change $\phi: T \to S$ is called \emph{faithful} for $X \to S$ if the cartesian square
	\begin{equation}\label{eqn:fiber}
	\begin{tikzcd}
	X_T = X \times_S T  \arrow{d}{}[swap]{a_T}  \arrow{r}{\phi_T} & X \arrow{d}{a} \\
	T\arrow{r}{\phi}  & S
	\end{tikzcd}
	\end{equation}
	is {\em Tor-independent}, which is equivalent to the condition that the natural transformation $ a^* \circ \phi_* \to \phi_{T *} \circ a_T^*: D(T) \to D(X)$ is an isomorphism.
	
	 Let $\shA \subset D(X)$ be $S$-linear and $\phi: T \to S$ be a projective faithful base-change. Then the {\em base-change of $\shA$ along $\phi$}, which is the $T$-linear admissible subcategory
	$$\shA_T  := \{ C \in D(X_T) \mid \phi_{T*} (C \otimes a_T^* F) \in \shA, \quad \forall F \in \Perf(T)\} \subset D(X_T),$$
see {\cite[Cor. 5.7]{Kuz11Bas}}. It satisfies $\phi_T^*(a) \in \shA_{T}$ for any $a \in \shA$, and $\phi_{T*}(b) \in \shA$ for $b \in \shA_{T}$ with proper support over $X$. The construction of base-change category $\shA_T$ is compatible for composition of base-changes $T' \to T \to S$, then $(\shA_T)_{T'} = \shA_{T'}$, see \cite[Lem. 2.7]{JL18Bl}.

Let $\shA \subset D(X)$ and $\shB \subset D(Y)$ be $S$-linear subcategories, and the fiber product $X\times_S Y$ of $S$-schemes $X,Y$ is Tor-independent. Then following Kuznetsov \cite{Kuz11Bas}, we can define the {\em exterior product of $\shA$ and $\shB$ over $S$} using base-change of categories:  
	$$\shA \boxtimes_S \shB : = \shA_Y \cap \shB_X \subset D(X \times_S Y),$$
where $\shA_Y$ is the base-change category of $\shA \subset D(X)$ along $Y \to S$, and $\shB_X$ is the base-change category of $\shB \subset D(Y)$ along $X \to S$. If $S=B$ is our fixed base scheme, we will omit the subscript $B$ and simply write $\shA \boxtimes \shB$.
Notice if a base-change $\phi \colon T \to S$ is faithful for $X$, then by definition the two constructions agrees:
	$$\shA \boxtimes_{S} D(T) = \shA_T \subset D(X_T).$$

\begin{lemma}[Associativity, see {\cite[Lem. 2.13]{JL18Bl}}] \label{lem:ass:tensor} Assume $X,Y,Z$ are $S$-schemes such that the fiber squares for fiber products $X \times_S Y$, $Y \times_S Z$, $X \times_S Z$ are all Tor-independent. Let $\shA \subset D(X)$, $\shB \subset D(Y)$, $\shC \subset D(Z)$ be $S$-linear admissible subcategories. Then there is a canonical identification of subcategories
	\begin{align*}
	 (\shA \boxtimes_S \shB) \boxtimes_S \shC = \shA \boxtimes_S (\shB \boxtimes_S \shC) \subset D(X \times_S Y \times_S Z).
	\end{align*}
\end{lemma}

\begin{lemma}[Compatibility, see {\cite[Lem. 2.14]{JL18Bl}}] \label{lem:tensor:bc} Let $S_k, T_k$ be schemes, $k=1,2$, $X_k$ be $S_1 \times T_k$ schemes, $\shA_k \subset D(X_k)$ be $S_1 \times T_k$-linear admissible subcategory, $Y_k$ be $S_2 \times T_k$-schemes, $\shB_k \subset D(Y_k)$ be $S_2 \times T_k$-linear admissible subcategory. Assume that the fiber squares for the fiber products $X_1 \times_{S_1} X_2$, $Y_1 \times_{S_2} Y_2$, $X_1 \times_{T_1} Y_1$, $X_2 \times_{T_2} Y_2$, and for
	$$ (X_1 \times_{S_1} X_2) \times_{T_1 \times T_2} Y_1 (\times_{S_2} Y_2) = (X_1 \times_{T_1} Y_1) \times_{S_1 \times S_2} (X_2 \times_{T_2} Y_2) = : Z$$
are all Tor-independent. Then there is a canonical identification of  subcategories:
	$$(\shA_1 \boxtimes_{S_1} \shA_2) \boxtimes_{T_1 \times T_2} (\shB_1 \boxtimes_{S_2} \shB_2) = (\shA_1 \boxtimes_{T_1} \shB_1) \boxtimes_{S_1 \times S_2} (\shA_2 \boxtimes_{T_2} \shB_2) \subset D(Z).$$
\end{lemma}
 
The $S$-linear categories behave well under base-change and exterior products.

\begin{proposition} [{\cite[Thm. 5.6]{Kuz11Bas}}]\label{prop:bcsod} If $f:X \to S$ is a morphism, $D(X) = \langle \shA_1, \ldots, \shA_n \rangle$ is a $S$-linear semiorthogonal decompositions by admissible subcategories, such that the projection $D(X) \to \shA_k$ is of finite cohomological amplitude, for $k=1,\ldots,n$. Let $\phi: T \to S$ be a faithful base-change for $f$, then there is a $T$-linear semiorthogonal decomposition
	$$D(X_T) = \langle \shA_{1T}, \ldots, \shA_{nT} \rangle$$
where $\shA_{kT}$ is the base-change category of $\shA_k$ along $T \to S$.
\end{proposition}

\begin{proposition}[{\cite[\S 5.5]{Kuz11Bas}}] \label{prop:product} Let $X,Y$ be $S$-schemes, and $\shA \subset D(X)$, $\shB \subset D(Y)$ be $S$-linear admissible subcategories, with $S$-linear semiorthogonal decompositions $\shA = \langle \shA_1, \ldots, \shA_m\rangle$ and $\shB = \langle \shB_1, \ldots, \shB_n \rangle$, $m,n \ge 1$. Assume the following technical condition holds: the projection functors $D(X) \to \shA$, $D(X) \to {}^\perp \shA$, $D(Y) \to \shB$, $D(Y) \to {}^\perp\shB$, $\shA \to \shA_i$, $\shB \to \shB_j$ are all of finite cohomological amplitudes. Assume the square for fiber product $X \times_S Y$ is Tor-independent. Then there is an $S$-linear semiorthogonal decomposition
	$$\shA \boxtimes_S \shB = \big\langle \shA_i \boxtimes_S \shB_j \big\rangle_{1\le i \le m, 1 \le j \le n},$$
where the order of the semiorthogonal sequence is any order $\{(i,j)\}$ extending the natural partial order of $\{i \mid 1 \le i \le m\}$ and $\{j \mid 1 \le j \le n\}$.
\end{proposition}

\subsection{Geometry of linear categories} \label{sec:geometric_construction}
In this section we review the constructions and results for basic geometric operations (projective bundle, generalized universal hyperplane and blowing up) on linear categories. Reference is \cite[\S 3]{JL18Bl}. The readers who are only concerned with schemes or have faith that the theorems for derive categories of schemes should also hold for reasonable subcategories, may skip this section.

\subsubsection{Projective bundle} \label{sec:app:proj_bd} Let $S$ be a smooth $B$-scheme, and $E$ be a vector bundle of rank $r$ on $S$, and denote $\pi \colon \PP_S(E) \to S$ the projection. Let $X$ be a proper $S$-scheme, $i_{\shA} \colon \shA \hookrightarrow D(X)$ be an inclusion of $S$-linear admissible subcategory. Then the {\em projective bundle} $\PP_\shA (E)$ of vector bundle $E$ over $\shA$ is the $S$-linear category defined by base-change:
	$$\PP_\shA (E) := \shA_{\PP_S(E)} = \shA \boxtimes_{S} D(\PP_S(E)) \subset  D(\PP_{X}(E)).$$
A $S$-linear semiorthogonal decomposition $D(X) = \langle \shA, \shB\rangle$ induces canonically a $S$-linear decomposition:
	$D(\PP_X(E)) = \langle \PP_\shA(E), \PP_\shB(E) \rangle.$
Notice the adjoint functors $\pi^* \colon D(X) \rightleftarrows D(\PP_{X}(E)) \colon \pi_*$ induce adjoint functors 
	$\pi^* \colon \shA \rightleftarrows \PP_{\shA}(E) \colon \pi_*,$ still denoted by same notations by abuse of notations.  

\begin{theorem}[Orlov's projective bundle formula \cite{O92}; see also {\cite[Thm. 3.1]{JL18Bl}}] \label{thm:app:pr_bd} The functors $\pi^*(-) \otimes \sO(k) \colon \shA \to \PP_{\shA}$ is fully faithful, $k \in \ZZ$, and the images induce a $S$-linear semiorthogonal decomposition
	$$\PP_\shA = \langle \pi^*\shA, \pi^*\shA \otimes \sO(1), \ldots, \pi^*\shA \otimes \sO(r-1)\rangle,$$
where $\sO(k)$ denotes the pull-back of line bundle $\sO_{\PP_S(E)}(k)$.
\end{theorem}

\subsubsection{Blowing up}\label{sec:bl} Let $S$ be a smooth $B$-scheme, and $i \colon Z \hookrightarrow S$ is a smooth codimension $r \ge 2$ local complete intersection subscheme, with normal bundle $N_i$. Denote $\widetilde{S} = \Bl_Z S$ the blowing up of $S$ along $Z$, $E_Z = \PP(N_i) \subset \widetilde{S}$ the exceptional divisor. Assume $X$ be a smooth proper $S$-scheme and $X_Z := X \times_S Z$ is of expected dimension $\dim X - r$, therefore $X_Z \subset X$ is local complete intersection of codimension $r$. Denote the blowing up of $X$ along $X_Z$ by $\beta \colon \widetilde{X} = \Bl_{X_Z} X \to X$, and the inclusion of exceptional divisor by $j \colon E_{X_Z} = \PP_{X_Z}(N_i) \hookrightarrow \widetilde{X}$, and $p\colon E_{X_Z} \to X_Z$ the projection. Let $\shA \subset D(X)$ be an $S$-linear admissible subcategory, the {\em blowing up category of $\shA$ along $\shA_Z$} is defined to be:
	$$\widetilde{\shA} := \shA \boxtimes_S D(\widetilde{S}) \subset D(\widetilde{X}) \quad \text{where} \quad \shA_Z := \shA \boxtimes_S D(Z) \subset D(X_Z).$$
Any $S$-linear semiorthogonal decomposition $\shA = \langle \shA_1, \shA_2 \rangle$ induces $S$-linear semiorthogonal decomposition $\widetilde{\shA} = \langle \widetilde{\shA_1}, \widetilde{\shA_2} \rangle$. Note that the projective bundle category $\PP_{\shA_Z}(N_i) \subset D(\PP_{X_Z}(N_i))$ plays the role exceptional divisors of the blowing-up, and is equipped with functors:
	$$p^* \colon \shA_Z \rightleftarrows \PP_{\shA_Z}(N_i) \colon p_*, \qquad j^* \colon \widetilde{\shA} \rightleftarrows   \PP_{\shA_Z}(N_i) \ \colon j_*,$$
induced from the functors on ambient spaces, and still denoted by same notations.

\begin{theorem}[Orlov's blowing up formula \cite{O92}; see also {\cite[Thm. 3.3]{JL18Bl}}] \label{thm:bl} The $S$-linear functors $\beta^*\colon \shA \to \widetilde{\shA}$ and $\Psi_k = j_* \, p^* (-) \otimes \sO_{\PP(N_i)}(k) \colon \shA_Z  \to  \widetilde{\shA}$ are fully faithful, $k \in \ZZ$, and their images induce $S$-linear semiorthogonal decompositions:
	\begin{align*} \widetilde{\shA}  & = \langle \beta^* \shA, ~ (\shA_Z)_0, (\shA_Z)_1, \ldots, (\shA_Z)_{r-2} \rangle 
				= \langle (\shA_Z)_{1-r}, \ldots, (\shA_Z)_{-2}, (\shA_Z)_{-1}, ~ \beta^* \shA\rangle,	
	\end{align*}
	where $(\shA_Z)_k$ denotes the image of $\shA_Z$ under $\Psi_k$, $k \in \ZZ$.
\end{theorem}

\subsubsection{Generalized universal hyperplane} \label{sec:hyp} Let $S$ be a smooth $B$-scheme, $i \colon Z \hookrightarrow S$ smooth subscheme, and assume further $Z = Z(s)$ is the zero locus of a regular section $s \in \Gamma(S,E)$ for a vector bundle $E$ of rank $r$. Then the section $s$ under the identification 
	$$H^0(\PP_S(E^\vee), \sO_{\PP(E^\vee)}(1)) = H^0(S, E),$$
 corresponds canonically a section $\tilde{s}$ of the line bundle $\sO_{\PP_S(E^\vee)}(1)$ on $\PP_S(E^\vee)$. Then the hypersurface $\shH_s : = Z(\tilde{s}) \subset \PP_S(E^\vee)$ is called {\em generalized universal hyperplane}.
Denote $\pi \colon \shH_s \to S$ the projection, then $\pi$ is a $\PP^{r-2}$-projective bundle  over $S\, \backslash \,Z$, and $\shH_s|_{\pi^{-1}(Z)} = \PP_Z(E^\vee|_Z) = \PP_Z(N_i^\vee)$, where $N_i$ is the normal bundle of $Z \subset S$ as usual. 

Let $a_X \colon X \to S $ be a smooth proper $S$-scheme such that $X_Z = X \times_S Z$ is of expected dimension $\dim X -r$. Then $X_Z$ is also cut out by the section $a_X^{*}\,s \in H^0(X, a_X^*E)$. Therefore we can similarly form  the generalized universal hyperplane $\shH_{X,s} \subset \PP_X(E)$ for $X$ with respect to the bundle $a_X^*E$ and section $a_X^*\,s$. By abuse of notation we will denote the bundle $a_X^* E$ and section $a_X^*\,s$ on $X$ still by $E$ and $s$. Denote the inclusions by $i\colon X_Z \hookrightarrow X$, $j \colon \PP_{X_Z}(N_i^\vee) \hookrightarrow \shH_{X,s}$, and the projections by $\rho\colon \PP_{X_Z}(N_i^\vee) \to X_Z$, $\pi\colon \shH_{X,s} \to X$. 

Let $\shA \subset D(X)$ be an admissible $S$-linear subcategory, then the {\em generalized universal hyperplane} for $\shA$ (with respect to vector bundle $E$ and regular section $s$) is defined to be 
	$$\shH_{\shA,s} := \shA \boxtimes_S D(\shH_s ) \subset D(\shH_{X,s}).$$
We will also write $\shH_{\shA} = \shH_{\shA,s}$ if there is no confusion. The functors on ambient spaces induce commutative diagrams of $S$-linear functors on the corresponding subcategories constructed from $\shA$, and we still denote by same notations, by abuse of notations.

\begin{theorem}[Orlov's generalized hyperplane theorem {\cite[Prop. 2.10]{O}}; see also {\cite[Thm. 3.9]{JL18Bl}}] \label{thm:hyp} The functors $j_*\,\rho^* \colon \shA_Z  \to \shH_{\shA,s}$ and $ \pi^*(-) \otimes \sO_{\PP(E^\vee)}(k) \colon \shA  \to \shH_{\shA,s} $ are fully faithful, $k \in \ZZ$, and there is $S$-linear semiorthogonal decompositions:
	\begin{align*}
	\shH_{\shA,s} & = \langle  j_* \, \rho^* \shA_Z,  ~~\pi^* \shA \otimes \sO_{\PP(E^\vee)}(1), \ldots , \pi^* \shA \otimes  \sO_{\PP(E^\vee)}(r-1)\rangle \\
	& = \langle \pi^* \shA \otimes \sO_{\PP(E^\vee)}(2-r), \ldots,  \pi^* \shA \otimes \sO_{\PP(E^\vee)}, ~~  j_* \, \rho^* \shA_Z \rangle.
	\end{align*}
\end{theorem}


\subsection{Lefschetz categories} \label{sec:lef} Lefschetz categories are the key ingredients for HPD theory. References are \cite{Kuz07HPD}, and \cite{Kuz08Lef, JLX17, P18, KP18, JL18Bl}. Let $\shA \subset D(X)$ be a fixed $\PP(V)$-linear admissible subcategory of a $\PP(V)$-scheme $X$, then a {\em right Lefschetz decomposition} of $\shA$ with respect to $\sO_{\PP(V)}(1)$ is a semiorthogonal decomposition of the form:
	\begin{equation}\label{lef:A} 
	\shA = \langle \shA_0, \shA_1(1), \ldots, \shA_{m-1}(m-1)\rangle,
	\end{equation}
	with $\shA_0 \supset \shA_1 \supset \cdots \supset \shA_{m-1}$ a descending sequence of admissible subcategories, where for a subcategory $\shA_{*} \subset \shA$,  $\shA_{*}(k) = \shA_{*} \otimes \sO_{\PP(V)}(k)$ denotes its image under the autoequivalence $\otimes \sO_{\PP(V)}(k)$, for $k \in \ZZ$. Dually, a {\em left Lefschetz decomposition} of $\shA$ is a semiorthogonal decomposition of the form:
	\begin{equation}\label{duallef:A} 
	\shA = \langle \shA_{1-m} (1-m), \ldots, \shA_{-1}(-1) , \shA_{0}\rangle,
	\end{equation}	
	with $\shA_{1-m} \subset \cdots \subset \shA_{-1} \subset \shA_{0}$ an ascending sequence of admissible subcategories. 
	
A $\PP(V)$-linear category $\shA$ is called a {\em Lefschetz category}, or is said to have a {\em Lefschetz structure}, if it is equipped with both a right and a left Lefschetz decomposition (with same $\shA_0$ and $m$) as above. If $D(X)$ has a Lefschetz structure, where $X$ is a $\PP(V)$-variety, then $X$ is called a {\em Lefschetz variety}. The number $m$ is called the {\em length} of $\shA$, and $\shA_0$ is called the {\em center} of $\shA$.  A Lefschetz decomposition for $\shA$ is totally determined by its center $\shA_0$ via relations $\shA_k = {}^\perp (\shA_0(-k)) \cap \shA_{k-1}$ and  $\shA_{-k} = (\shA_0(k))^\perp \cap \shA_{1-k}$, where $k=1,2,\ldots, m-1$, see \cite[Lem. 2.18]{Kuz08Lef}. See references above for more properties of Lefschetz categories.

For a Lefschetz category $\shA$ as above, denote $\foa_k: = \shA_{k+1}^\perp \cap \shA_k$, the right orthogonal of $\shA_{k+1}$ inside $\shA_{k}$ for $0 \le k \le m-1$. Then the admissible subcategories $\foa_k$'s are the {\em primary components} of $\shA$. For $0 \le k \le m-1$. It holds that 
	$$\shA_k = \langle \foa_{k}, \shA_{k+1} \rangle = \langle \foa_{k}, \foa_{k+1}, \ldots, \foa_{m-1}\rangle.$$
Dually let $\foa_{-k}: = {}^\perp \shA_{-k-1} \cap \shA_{-k}$ be the left orthogonal of $\shA_{-k-1}$ inside $\shA_{-k}$ for $0 \le k \le m-1$. Then $\foa_{-k}$ are also admissible subcategories and for $0 \le k \le m-1$,
	$$\shA_{-k} =  \langle \shA_{-(k+1)} , \foa_{-k}\rangle =  \langle \foa_{1-m}, \ldots, \foa_{-1-k}, \foa_{-k} \rangle.$$
	
In \cite{Kuz07HPD} it is required that a $\PP(V)$-linear Lefschetz category $\shA$ should satisfies
	$$ {\rm length} (\shA) < \rank V.$$
This condition is called {\em moderate} condition in \cite{P18}. In this paper we follow  \cite{Kuz07HPD} and require all Lefschetz categories to be {\em moderate}. In fact, a non-moderate Lefschetz category can always be refined to be a moderate one, see \cite[Lem. 2.22]{JL18Bl}.

The following criterion for equivalence of Lefschetz categories is useful:
\begin{lemma}[{\cite[Lem. 2.14]{KP18}}] \label{lem:equiv-lef} Let $\phi \colon \shA \to \shB$ is a $\PP(V)$-linear functor between two $\PP(V)$-linear Lefschetz categories $\shA$ and $\shB$ with Lefschetz centers $\shA_0$ and respectively $\shB_0$. Assume $\phi$ admits left adjoint $\phi^* \colon \shB \to \shA$. If $\phi$ induces $\shA_0 \simeq \shB_0$, $\phi^*$ induces $\shB_0 \simeq \shA_0$, then $\phi$ is an equivalence of Lefschetz categories. 
\end{lemma}
\begin{proof}The fully faithfulness of $\phi$ follows directly from $\cone (\phi^*\circ \phi \to \id)$ is zero on $\shA_k \subset \shA_0$, hence on $\shA_k(kH)$ by $\PP(V)$-linearity, hence zero on $\shA$; the essential surjectivity follows from ${\rm Im} (\phi) \supset \shB_0$, and hence contains all $\shB_0(k H)$ by $\PP(V)$-linearity of $\phi$, $k \in \ZZ$, and therefore contains the whole $\shB$.
\end{proof}

\subsection{Homological projective duality}\label{sec:HPD}
Let $Q  = \{(x,[H]) \mid x \in H \}\subset \PP(V) \times \PP(V^\vee)$ be the universal quadric for $\PP(V)$ (or equivalently for $\PP(V^\vee)$). Then for a $\PP(V)$-variety $X$, the {\em universal hyperplane $\shH_{X}$ for $X \to \PP(V)$},  defined in introduction \S \ref{sec:intro:classical}, can also be defined as
	$$\shH_{X} \equiv \shH_{X \, /\PP(V)} : = X \times_{\PP(V)} Q \subset X \times \PP(V^\vee).$$
Denote the inclusion by $\delta_{\shH} \colon \shH_{X} \hookrightarrow X \times \PP(V^\vee)$. If $\shA \subset D(X)$ is a $\PP(V)$-linear admissible subcategory, then the {\em universal hyperplane} $\shH_{\shA}$ is defined to be:
	$$\shH_{\shA} \equiv  \shH_{\shA \, /\PP(V)}: = (\shA)_{\shH_X} = \shA \boxtimes_{\PP(V)} D(Q) \subset D(\shH_X),$$
where $(\shA)_{\shH_X}$ denotes the base-change of $\shA$ along $\shH_X \to \PP(V)$. 

Notice that $Q = \PP_{\PP(V)}(\Omega_{\PP(V)}^1)$ is a projective bundle, therefore the construction and results of projective bundles (see \S \ref{sec:app:proj_bd}) can be directly applied to
	$\shH_{\shA} = \PP_{\shA}(\Omega_{\PP(V)}^1).$
Notice also that the inclusion $\delta_{\shH}$ induces adjoint functors $\delta_{\shH}^* \colon \shA \boxtimes D(\PP(V^\vee)) \rightleftarrows \shH_{\shA} \colon \delta_{\shH *}$ as before.

\begin{definition} 
Let $\shA$ be a $\PP(V)$-linear Lefschetz category with Lefschetz center $\shA_0$. Then the {\em HPD category $\shA^\hpd$} of $\shA$ over $\PP(V)$ is defined to be
	$$\shA^\hpd  \equiv (\shA)_{/\PP(V)}^\hpd :=  \{C \in \shH_{\shA}  \mid \delta_{\shH*} \,C \in \shA_0 \boxtimes D(\PP(V^\vee))\} \subset \shH_{\shA}.$$
If $\shA = D(X)$ for a variety $X$ with $X \to \PP(V)$, and furthermore if there exists a variety $Y$ with $Y \to \PP(V^\vee)$, and a Fourier-Mukai kernel $\shP \in D(Y \times_{\PP(V^\vee)} \shH_X)$ such that the $\PP(V^\vee)$-linear Fourier Mukai functor $\Phi_{\shP}^{Y \to \shH_X} \colon D(Y) \to D(\shH)$ induces an equivalence of categories $D(Y) \simeq D(X)^\hpd$, then $Y$ is called the {\em HPD variety} of $X$.
\end{definition}

\begin{lemma}[{\cite{Kuz07HPD,JLX17,P18}}] \label{lem:sodH} There is a $\PP(V^\vee)$-linear semiorthogonal decomposition
	$$\shH_{\shA} = \big \langle \shA^\hpd, ~~\delta_{\shH}^* (\shA_1(H) \boxtimes D(\PP(V^\vee))), \ldots,  \delta_{\shH}^* (\shA_{m-1}((m-1)H) \boxtimes D(\PP(V^\vee))) \big \rangle.$$
\end{lemma}

The HPD is a {\em reflexive} correspondence between Lefschetz categories over $\PP(V)$ and $\PP(V^\vee)$. More precisely, the HPD category $\shA^\hpd$ is a Lefschetz category with respect to $\otimes \sO_{\PP(V^\vee)}(1)$, and there is a $\PP(V)$-linear equivalence of Lefschetz categories $\shA \simeq (\shA^\hpd)^\hpd$. See \cite{Kuz07HPD,JLX17,P18}.

\begin{lemma}[{\cite{Kuz07HPD,JLX17,P18}}] \label{lem:hpdcat} Denote $\gamma \colon \shA^\hpd \hookrightarrow \shH_\shA$ the inclusion, and $\shA^\hpd_0$ the Lefschetz center of the HPD category $\shA^\hpd$. Then for any $a \in \shA_0$, we have 
	$\pi_* \cone( \pi^* a \to \gamma \gamma^* \pi^* a) = 0.$
Furthermore, there are mutually inverse equivalences between Lefschetz centers:
	$$\pi_* \circ \gamma \colon \shA_0^\hpd \simeq \shA_0, \qquad \gamma^* \circ \pi^* \colon \shA_0 \simeq \shA_0^\hpd.$$
\end{lemma}

The fundamental result of HPD theory is the Kuznetsov's HPD theorem for linear sections \cite{Kuz07HPD}. Since we will not use this result, we refer the readers to the references \cite{Kuz07HPD,Kuz14SOD, RT15HPD, JLX17,P18, JL18Bl} for its precise statement and various applications.

\subsection{$n$-HPD category} \label{sec:nHPD}  The construction of last section can be generalized to more than one $X$. For varieties $X_k$ with morphisms $X_k \to \PP(V)$, where $k=1,2,\cdots n$, we can define the {\em $n$-universal hyperplane} for $X_k \to \PP(V)$ to be: 
	$$\delta_{\shH} \colon \shH(X_1, \ldots, X_n) = \{(x_1, \cdots, x_n,  [H]) \mid x_1,\ldots, x_n \in H \} \hookrightarrow \prod_{k=1}^{n} X_k \times \PP(V^\vee).$$
It is the zero loci of the canonical regular section $\sigma$ of the rank $n$ vector bundle $\bigoplus_{k=1}^n \sO(H_k+ H')$, where $\sigma$ is determined by morphisms $X_k \to \PP(V)$ under the identification
	$$H^0(\prod_{k=1}^n X_k \times \PP V, \bigoplus_{k=1}^n \sO(H_k+ H')) = \bigoplus_{k=1}^n H^0(X_k, V\otimes \sO(H_k)),$$ 
and $H_k$ is the hyperplane class for $X_k$, i.e. the pulling back of $\sO_{\PP(V)}(H)$, and $H'$ is the hyperplane class of $\PP(V^\vee)$. If $n=1$, $X_1=X$, then $\shH(X) = \shH_{X}$ is the usual universal hyperplane. The variety $Q_n = \shH(\PP(V), \ldots ,\PP(V))$ is called the {\em $n$-universal hyperplane}. Note that $Q_n$ is universal in the sense that $\shH(X_1, \ldots, X_n) = (X_1 \times \cdots \times X_n)  \times_{\PP(V) \times \cdots \PP(V)} Q_n$ for all $X_k \to \PP(V)$. Note that $Q_1 = Q  \subset \PP(V) \times \PP(V^\vee)$ is nothing but the usual universal quadric.

Let $\shA^{(k)} \subset D({X_k})$ be $\PP(V)$-linear admissible subcategories, $k=1,\ldots,n$, where ${X_k} \to \PP(V)$ are smooth varieties. Then the {\em $n$-universal hyperplane category $\shH(\shA^{(1)},\ldots, \shA^{(n)})$} for the $\PP(V)$-linear categories $\shA^{(k)}$ is the category
	\begin{equation}
	\label{eq:def:n-Hyp}
	\shH(\shA^{(1)},\ldots, \shA^{(n)})  : = (\shA^{(1)} \boxtimes \cdots \boxtimes \shA^{(n)}) \boxtimes_{~\PP(V) \times \cdots \times \PP(V)} D(Q_n).
	\end{equation}

In the following we simply denote $\shH = \shH(X_1, \ldots, X_n)$, and recall  $\delta_{\shH}: \shH \hookrightarrow \prod_{k=1}^{n} X_k \times \PP(V^\vee)$ is the inclusion. Denote $\pi: \shH \to \prod X_k$ and $h: \shH \to \PP(V^\vee)$ the projections. 
Then we have a diagram of $\PP(V)$-linear functors:
	$$
	\begin{tikzcd}
	(\shA^{(1)} \boxtimes \cdots \boxtimes \shA^{(n)})  \boxtimes  D( \PP(V^\vee)) \ar[shift left]{d}{\delta_{\shH *}}  	\ar[hook]{r}{\gamma_{{}_\shA} \times \Id} 	&	D(X_1 \times \cdots \times X_n) \boxtimes D( \PP(V^\vee)) \ar[shift left]{d}{\delta_{\shH *}}  \\
	\shH(\shA^{(1)},\cdots, \shA^{(n)}) \ar[shift left]{d}{\pi_*} \ar[hook]{r}{\gamma_{{}_\shH}} 	\ar[shift left]{u}{\delta_{\shH}^*}	& 		D(\shH)  \ar[shift left]{d}{\pi_*} \ar[shift left]{u}{\delta_{\shH}^*}	\\
	\shA^{(1)} \boxtimes \ldots \boxtimes \shA^{(n)}  \ar[shift left]{u}{\pi^*} \ar[hook]{r}{\gamma_{{}_\shA}} & D(X_1  \times \cdots \times X_n) \ar[shift left]{u}{\pi^*},
	\end{tikzcd}
	$$
which is commutative for all pushforward functors and respective for all pullback functors. Here by abuse of notations we denote the restrictions of $\pi_*$, $\pi^*$, $\delta_{\shH *}$, $\delta_{\shH}^*$ by the same notations; and $\gamma_{\shA}$ and $\gamma_{\shH}$ denote the natural inclusion functors as usual.

Denote by $\bigotimes_{k=1}^n \shA^{(k)} = \shA^{(1)} \boxtimes \cdots \boxtimes \shA^{(n)}$.
\begin{lemma} \label{lem:char:nHA} In the same situation as above, we have the following characterizations:
	 \begin{align*}
	\shH(\shA^{(1)},\ldots, \shA^{(n)})  &=  \{C \in D({\shH}) \mid \delta_{{\shH}*} \,C \in \bigotimes_{k=1}^n \shA^{(k)} \boxtimes D(\PP(V^\vee))\} \subset D({\shH}) \\
		&= \{C \in D({\shH}) \mid \pi_* (C \otimes h^* F) \in \bigotimes_{k=1}^n \shA^{(k)}, ~~\forall F \in D(\PP(V^\vee)) \} \subset D({\shH}) .
	 \end{align*}
\end{lemma}
\begin{proof} By definition of base-change of categories, for $C \in D({\shH})$, $C \in \shH(\shA^{(1)},\ldots, \shA^{(n)})$ if and only if $\pi_{*}(C\otimes G) \in  \bigotimes_{k=1}^n \shA^{(k)}$ for all $G\in D(Q_n)$. However, as a projective bundle $Q_n \to \PP(V^\vee)$, $D(Q_n)$ is generated by elements $\delta_{Q_n}^* (E_1 \otimes \cdots \otimes E_n \otimes F)$, for $E_k \in \PP(V)$ and $F \in \PP(V^\vee)$. Then from $\delta_{Q_n}^*(E_1 \otimes \cdots \otimes E_n \otimes F) = \pi^* (E_1 \otimes \cdots E_n) \otimes h^* F$, and $\shA^{(k)}$'s are $\PP(V)$-linear, one obtains the desired characterization of $ \shH(\shA^{(1)},\ldots, \shA^{(n)})$.
\end{proof}

\begin{definition} \label{def:nHPD} Let $X_k \to \PP(V_k)$ be smooth varieties, $k=1,\cdots, n$ and $\shA^{(k)} \subset D({X_k})$ be $\PP(V)$-linear Lefschetz categories with Lefschetz center $\shA_0^{(k)}$. The {\em $n$-HPD category} $\shC$ for $\shA^{(1)},\ldots, \shA^{(n)}$ is the full $\PP(V^\vee)$-linear subcategory defined by
	$$\shC  =  \{C \in  \shH(\shA^{(1)},\ldots, \shA^{(n)})  \mid \delta_{{\shH}*} \,C \in \bigotimes_{k=1}^n \shA^{(k)}_0 \boxtimes D(\PP(V^\vee))\} \subset \shH(\shA^{(1)},\ldots, \shA^{(n)}).$$
\end{definition}

\begin{remark} \label{rmk:nHPD-char} From definition of base-change, the $n$-HPD category $\shC$ as a subcategory of $D(\shH) = D(\shH(X_1,\ldots, X_n))$ can also be characterized by:
 \begin{align*}
	 \shC &= 	\{C \in D({\shH}) \mid \pi_* (C \otimes h^* F) \in \bigotimes_{k=1}^n \shA^{(k)}_0 \text{~for all~} F \in D(\PP(V^\vee)) \}  \subset D({\shH}) \\
	    &=  \{C \in D({\shH}) \mid \delta_{{\shH}*} \,C \in \bigotimes_{k=1}^n \shA^{(k)}_0 \otimes D(\PP(V^\vee))\} \subset D({\shH}).
 \end{align*}
\end{remark}

 The following is the generalization of Lem. \ref{lem:sodH} to $n$-HPD category. To state the result we introduce following notations. For any subset $I \subsetneqq \{1,2,\cdots,n\}$, denote by $|I|$ its cardinality, and by $I^c$ its complement. Denote $\shH_I$ the $|I|$-universal hyperplane for $\{X_k\}_{k \in I}$, and $\shC_I$ the $|I|$-HPD category for $\shA^{(k)} \subset D({X_k})$, $k \in I$. Note for every $I$ the inclusion $\delta_{\shH}$ factors through an inclusion $\iota_I: \shH \hookrightarrow \shH_I  \times \prod_{k \in I^c} X_k \subset \prod_{k=1}^n X_k \times \PP(V)$. For $I = \emptyset$ the empty set, we will use the convention: $\iota_{\emptyset} = \delta_{\shH}: \shH \hookrightarrow  \prod_{k=1}^n X_k \times \PP(V)$ and $\shC_{\emptyset} \equiv D(\PP V^\vee)$. Notice that for $I = \{k\}$, then $\shC_I = \shA^{(k), \hpd}$, the (usual) HPD category of $\shA^{(k)}$.
 
 \begin{lemma}\label{lem:sod:nH} Let $\shA^{(k)} \subset D({X_k})$ be $\PP(V)$-linear Lefschetz categories, with Lefschetz center $\shA_0^{(k)}$ and Lefschetz components $\shA^{(k)}_{i_k}$, $i_k \in \ZZ$. Then for any $n \ge 1$, $I  \subsetneqq  \{1,\cdots,n\}$, the functors $\iota_I^*$ is fully faithful on the subcategories 
 	$$\shC_I \boxtimes  \bigotimes_{k \in I, i_k \ge 1} \shA_{i_k}^{(k)} (i_k H_k)  \subset \shH_I  \boxtimes \bigotimes_{k \in I^c} \shA^{(k)},$$
and the images induce a $\PP(V^\vee)$-linear semiorthogonal decomposition:
	$$ \shH(\shA^{(1)},\ldots, \shA^{(n)})   = \Big \langle \shC,  \big \langle  \iota_I^* \big(\shC_I \boxtimes  \bigotimes_{k \in I, i_k \ge 1} \shA_{i_k}^{(k)} (i_k H_k) \big)\big \rangle_{I  \subsetneqq  \{1,\cdots,n\}} \Big \rangle,$$
 where $\shC$ is the $n$-HPD category for $\{\shA^{(k)}\}_{k=1,\ldots,n}$, and $\shC_I$ the $|I|$-HPD category for $\{\shA^{(k)}\}_{k \in I}$. The order of the above semiorthogonal sequence is any order that is compatible with the (reversed) partial order of all subsets $\{I  \subseteq \{1,2, \cdots, n\}\}$, where we regard $ \shC_{ \{1,2, \cdots, n\}}=\shC $.
\end{lemma}

The proof is straight forward and will be given in appendix \S \ref{app:proof:sod:nH}.

\subsection{HPD with base-locus} \label{sec:app:HPDbs} 
	\subsubsection{Refined blow-up along base-locus} \label{sec:HPDbs:ref-Bl} Let $X$ be a smooth scheme with $X \to \PP(V)$, $L \subset V^\vee$ be a linear subbundle of rank $\ell$ over $B$ and $L^\perp = \Ker (V \to L^\vee) \subset V$ be the orthogonal bundle. Assume the base-locus $X_{L^\perp} = X \times_{\PP(V)} \PP(L^\perp)$ (of the linear system $L$) is of expected dimension $\dim X - \ell$. Denote $\widetilde{X} = \Bl_{X_{L^\perp}} X$ the blowing up of $X$ along $X_{L^\perp}$. By construction $\widetilde{\PP(V)}\subset S\times \PP(L^\vee)$ admits a map to $\PP(L^\vee)$ by projection. 

Let $\shA \subset D(X)$ be a $S=\PP(V)$-linear Lefschetz subcategory with Lefschetz center $\shA_0$ and length $m$. Apply the construction of \S \ref{sec:bl} to $\shA$, consider the blowing up category 
	$$\widetilde{\shA} := \shA \boxtimes_{\PP(V)} D(\widetilde{\PP(V)}) \subset D(\widetilde{X})$$
of $\shA$ along $\shA_{\PP(L^\perp)}$, where 
	$\shA_{\PP(L^\perp)} = \shA \boxtimes_{\PP(V)} D(\PP(L^\perp))$. The blowing-up category $\widetilde{\shA}$ is equipped with a $\PP(L^\vee)$-linear structure from the projection $\widetilde{\PP(V)} \to \PP(L^\vee)$. From Thm. \ref{thm:bl} there is a fully faithful functor $\beta^* \colon \shA \to \widetilde{\shA}$.

\begin{lemma}[{\cite[Lem. 4.1]{JL18Bl}}] \label{lem:app:ref-bl} The action functor $\act \colon \widetilde{\shA} \boxtimes D(\PP(L^\vee)) \to \widetilde{\shA}$ (for the $\PP(L^\vee)$-linear structure of the category $\widetilde{\shA}$) is fully faithful on the subcategories of $\widetilde{\shA} \boxtimes D(\PP(L^\vee))$:
	$$\beta^*(\shA_{\ell-1})\boxtimes D(\PP(L^\vee)), \beta^*(\shA_{\ell}(1))\boxtimes D(\PP(L^\vee)) \ldots, \beta^*(\shA_{m-1}(m-\ell))\boxtimes D(\PP(L^\vee)),$$
and their images remain a semiorthogonal sequence in $\widetilde{\shA}$. 
\end{lemma}

The right orthogonal of the images of above lemma inside $\widetilde{\shA}$ is called {\em refined blow-up category} of $\shA$ in {\cite{JL18Bl}}, and denoted by $\rBlA$.

\begin{proposition}[{\cite[Prop. 4.4]{JL18Bl}}] \label{prop:app:lef:ref-bl}  $\rBlA$ admits a (moderate) $\PP(L^\vee)$-linear Lefschetz structure, with Lefschetz components $\rBlA_k  :=\langle \beta^* \shA_k', (\shC_L)_0 \rangle$, for $0 \le k \le \ell-2$, 
where $\shA_k': = \shA_{\ell-1}^\perp \cap \shA_k \subset \shA_0,$
and $\shC_L$ is the essential component of $\shA_{L^\perp}$ defined by
	\begin{align*} 
	\shC_L &= \big\langle i_L^* \shA_{\ell}(1), \ldots, i_L^* \shA_{m-1}(m-\ell) \big \rangle^\perp \subset \shA_{L^\perp} \\
	& = \{ C \in \shA_{L^\perp} \mid i_{L*} C \in \langle \shA_0(1-\ell), \shA_{1}(2-\ell) \ldots, \shA_{\ell-1} \rangle \subset \shA\},
	\end{align*}
where $i_L \colon X_{L^\perp} \to X$ is the inclusion, and $(\shC_L)_0$ denotes  $j_*\, p^*(\shC_L)$ as in Thm. \ref{thm:bl}.
\end{proposition}

\subsubsection{Generalized universal hyperplane and HPD} \label{sec:HPDbs:hyp} Next we show that the HPD of $\rBlA$ over $\PP(L^\vee)$ is simply given by the linear section $(\shA^\hpd)|_{\PP(L)}$ of $\shA^\hpd$, which can also be intrinsically defined as the nontrivial component of the generalized universal hyperplane of $\shA$ over $\PP(L)$. More precisely, although $X \dasharrow \PP(L^\vee)$ is only a rational map, we are actually in the situation of \S \ref{sec:hyp}, where $S = \PP(V)$, $Z = \PP(L^\perp)$, $E = L^\vee \otimes \sO_{\PP(V)}(1)$, $i \colon \PP(L^\perp) \hookrightarrow \PP(V)$, and the section $s_L$ of $E$ is the canonical section which corresponds to the inclusion $L\subset V^\vee$ under the identification:
 	$$s_L \in \Gamma(\PP(V), L^\vee \otimes \sO_{\PP(V)}(1)) = \Hom_B(L, V^\vee).$$
Therefore we can form the {\em generalized universal hyperplane} for $X \to \PP(V) \dasharrow \PP(L^\vee)$:
	$$\shH_{X,L} := \shH_{X, s_L}  \hookrightarrow \PP_X(E) = X \times \PP(L),$$
where $\delta_{\shH_{L}} \colon \shH_{X,L} \hookrightarrow X \times \PP(L)$ is an inclusion of the divisor $\sO(1,1):=\sO_X(1) \boxtimes \sO_{\PP(L)}(1)$.

Let $\shA \subset D(X)$ be a $\PP(V)$-linear Lefschetz category of length $m$, with Lefschetz center $\shA_0$ and components $\shA_k$'s. Then by construction of \S \ref{sec:hyp} we obtain the {\em generalized universal hyperplane} $\shH_{\shA,L} \subset D(\shH_{X,L})$ for $\shA$, with induced adjoint functors:
		$$ \delta_{\shH_L}^* \colon \shA \boxtimes D(\PP(L)) \rightleftarrows \shH_{\shA,L}  \colon \delta_{\shH_L \, *}.$$

\begin{theorem}[{\cite[Thm. 4.5]{JL18Bl}}] \label{thm:app:HPDbs} Consider the full $\PP(L)$-linear subcategory of the generalized universal hyperplane $\shH_{\shA,L}$ defined by:
	$$\shD_L := \{C \in \shH_{\shA,L}  \mid \delta_{\shH_L \,*} \,C \in \shA_0 \boxtimes D(\PP(L))\} \subset \shH_{\shA,L}.$$
Then $\shD_L$ is the HPD category of the Lefschetz category $\rBlA$ over $\PP(L^\vee)$: $\shD_L \simeq (\rBlA)^\hpd_{/\PP(L^\vee)}$. Furthermore, $\shD_L$ is naturally $\PP(L)$-linear equivalent to $(\shA^\hpd)|_{\PP(L)}$, the base-change category of the HPD category $\shA^\hpd = (\shA)^\hpd_{/\PP(V)}$ along inclusion $\PP(L) \subset \PP(V^\vee)$.
\end{theorem}

Note that since HPD is a reflexive relation, above theorem gives a complete answer to the question ``what is the HPD of a linear section of a pair of HPDs".

\newpage
\section{Categorical joins}
\subsection{Classical joins} \label{sec:classical_joins}
 Let $X_k \to \PP(V)$, $k=1,2$ be to two smooth varieties with proper morphisms to $\PP(V)$. Denote the hyperplane class of $\PP(V)$ by $H$ and its pullback to $X_k$ by $H_k$. The {\em abstract join} (called {\em resolved join} in \cite{KP18}) {\em of $X_1$ and $X_2$} is defined to be:
	\begin{align}\label{eq:abstract-join}
	p: P(X_1,X_2) := \PP_{X_1 \times X_2}(\sO(-H_1) \oplus \sO(-H_2))  \to X_1 \times X_2,
	\end{align}
which is a $\PP^1$-bundle over $X_1 \times X_2$, with two canonical sections:
	\begin{align}
	\label{eq:ek}
	\varepsilon_k \colon E_k(X_1, X_2) = \PP_{X_1 \times X_2}(\sO(-H_k)) 			\hookrightarrow P(X_1,X_2),
	\quad 
	k = 1,2.
	\end{align}
Note that $p|_{E_k}: E_k(X_1, X_2) \simeq X_1 \times X_2$, and $P(X_1,X_2)$ is equipped with a natural morphism
	\begin{align} 
	\label{eq:f:P}
	f \colon P(X_1,X_2) \to \PP(V \oplus V)
	\end{align}
 induced from the inclusion $\sO(-H_1) \oplus \sO(-H_2) \subset (V \oplus V) \otimes \sO$ of vector bundles over $X_1 \times X_2$. 
The the {\em ruled join of $X_1$ and $X_2$} is defined to be image of $P(X_1,X_2)$ in $\PP(V\oplus V)$:
	$$R(X_1,X_2): = f(P(X_1,X_2)) \subset \PP(V\oplus V).$$

Denote $\Delta_{V} = \{(v,v) | v \in V\} \subset V\oplus V$ the diagonal, and consider the linear projection $\pi_{\Delta_V} \colon \PP(V\oplus V) \dasharrow \PP(V)$ of $\PP(V\oplus V)$ from diagonal $\PP(\Delta_V)$. The image of the ruled join $R(X_1,X_2)$ under the rational map $\pi_{\Delta_{V}}$ is called the {\em (classical) join of $X_1$ and $X_2$}:
		$$X_1 \join X_2 := \pi_{\Delta_{V}}(R(X_1,X_2)) \equiv \overline{\pi_{\Delta_{V}}(R(X_1,X_2) \backslash \PP(\Delta_V))} \subset \PP(V).$$
The classical join is also denoted by $J(X_1,X_2)$, but be cautious that some authors use $J(X_1,X_2)$ for the ruled join $R(X_1,X_2)$. The classical joins has the meaning that it is the closure of union of lines connecting points of $X_1$ and $X_2$ in $\PP(V)$.

We have a chain of (rational) maps between these joins:
		\begin{align} 
		\label{eq:chain_joins}
		P(X_1,X_2)  \to R(X_1,X_2)  \dashrightarrow X_1 \join X_2,
		\end{align}
where the latter is a morphism if and only if $f^{-1}(\PP(\Delta_V)) \simeq X_1 \times_{\PP(V)} X_2 = \emptyset$, where $f$ is the map in (\ref{eq:f:P}). If this happens, then the above chain of morphisms is the {\em Stein factorization} of $P(X_1,X_2) \to X_1 \join X_2$. More concretely, $P(X_1,X_2) \to R(X_1,X_2)$ is a birational contraction (which contracts divisors $E_k(X_1,X_2)$ to $f(X_k)$, $k=1,2$), and $R(X_1,X_2) \to X_1 \join X_2$ is finite morphism of degree $d \ge 1$. (Note that for a general point $x \in X_1 \join X_2$, $d$ is the number of lines of the form $\langle x_1,x_2\rangle$, where $x_k \in X_k$, $k=1,2$, passing through $x$.)

In general if $X_1 \times_{\PP(V)} X_2 \neq \emptyset$, in order to eliminate the indeterminacy of the rational map (\ref{eq:chain_joins}), we need to blow-up $P(X_1,X_2)$ along $f^{-1}(\PP(\Delta_V)) \simeq X_1 \times_{\PP(V)} X_2$:
	\begin{align} 
		\label{eq:Bl_P}
		\beta_P \colon \widetilde{P}(X_1,X_2) := \Bl_{f^{-1}(\PP(\Delta_V))} P(X_1,X_2) \to P(X_1,X_2).
	\end{align}	
Denote the inclusion of exceptional divisor by:
	\begin{align}
	\label{eq:e_P}
	\varepsilon_P \colon E_P = E_P(X_1,X_2) \hookrightarrow \widetilde{P}(X_1,X_2).
	\end{align}
Let $\widetilde{\PP(V\oplus V)}$ be the blowing-up of $\PP(V\oplus V)$ along diagonal $\PP(\Delta_V) \subset \PP(V\oplus V)$, then there is a morphism $\widetilde{\PP(V\oplus V)} \to \PP(V)$. By 
construction $\widetilde{P}(X_1,X_2)$ admits a morphism to $\PP(V)$ given by the composition:
	$$\tilde{f} \colon \widetilde{P}(X_1,X_2) \to \widetilde{\PP(V\oplus V)} \to \PP(V).$$
Since the sections $E_k(X_1, X_2)$ are disjoint from $f^{-1}(\PP(\Delta_V))$, therefore their strict transforms are just the inverse images; denote the inclusions of the strict transforms by 
	$$\widetilde{\varepsilon_k} \colon \widetilde{E_k}(X_1, X_2) = \beta_{P}^{-1}(E_k(X_1,X_2)) \hookrightarrow  \widetilde{P}(X_1,X_2),$$
then $\beta_{P}$ induces $\widetilde{E_k}(X_1, X_2)  \simeq E_k(X_1,X_2)$, $k=1,2$. Notice that the restrictions of (\ref{eq:f:P}):
	$$f|_{\widetilde{E_k}} \colon \widetilde{E_k}(X_1,X_2) \to \PP(V), \quad k=1,2$$
factor through the isomorphisms:
	\begin{align*}
	f|_{\widetilde{E_1}} \colon \widetilde{E_1}(X_1,X_2)  \to \widetilde{\PP}(V\oplus\{0\}) \xrightarrow{\sim} \PP(V), 
	\qquad f|_{\widetilde{E_2}} \colon  \widetilde{E_2}(X_1,X_2) \to \widetilde{\PP}(\{0\} \oplus V) \xrightarrow{\sim} \PP(V),
	\end{align*}
where $\widetilde{\PP}(V\oplus\{0\})$ and $\widetilde{\PP}(\{0\} \oplus V)$ are the proper transforms of $\PP(V\oplus\{0\})$ and respectively $\PP(\{0\} \oplus V)$ along the blowing up $\widetilde{\PP(V\oplus V)} \to \PP(V \oplus V)$. The image of $f|_{\widetilde{E_k}}$ coincides with the image of $X_k \to \PP(V)$, $k=1,2$.

\subsection{Categorical joins} \label{sec:cat_joins} Many constructions of previous sections can be carried out for categories by constructions of \S \ref{sec:geometric_construction}. Let $X_k \to \PP(V)$, $k=1,2$ to two smooth varieties with proper morphisms to $\PP(V)$. Assume that $X_1 \times_{\PP(V)} X_2$ is of expected dimension $\dim X_1 + \dim X_2 - \dim \PP(V)$ at every point. \footnote{~As usual, this Tor-independence condition can be removed if we consider derived intersection $X_1 \times_{\PP(V)} X_2$ rather than scheme-theoretic intersection}. 

Let $\shA^{(k)}\subset D(X_k)$ be $\PP(V)$-linear admissible subcategories. 
First, the {\em abstract join} of $\shA^{(1)}$ and $\shA^{(2)}$ can be defined to be: 
	$$P(\shA^{(1)}, \shA^{(2)}): = \PP_{\shA^{(1)} \boxtimes \shA^{(2)}}((\sO(-H_1) \oplus \sO(-H_2))) \subset D({P}(X_1,X_2)),$$
by projective bundle construction of \S \ref{sec:app:proj_bd}. Then $P(\shA^{(1)}, \shA^{(2)})$ is a $\PP^1$-bundle category over $\shA^{(1)} \boxtimes \shA^{(2)}$, with induced adjoint functors $p^* \colon \shA^{(1)} \boxtimes \shA^{(2)} \rightleftarrows P(\shA^{(1)}, \shA^{(2)}) \colon p_* $, where $p$ is the $\PP^1$-bundle map (\ref{eq:abstract-join}). The zero sections (\ref{eq:ek}) corresponds to admissible subcategories
	$$E_k(\shA^{(1)}, \shA^{(2)}) \subset D(E_k(X_1,X_2)), \qquad k=1,2,$$
with induced adjoint functors $\eps_k^* \colon P(\shA^{(1)}, \shA^{(2)})  \rightleftarrows E_k(\shA^{(1)}, \shA^{(2)}) \colon \eps_{k*}$, where $\varepsilon_k$ is the inclusion (\ref{eq:ek}).
Notice the isomorphism
	$$p \, \circ \, \eps_k \colon E_k(X_1,X_2) \simeq X_1 \times X_2$$
induces equivalences of categories
	$$\eps_{k}^* \circ p^*  \colon E_k(\shA^{(1)}, \shA^{(2)}) \simeq \shA^{(1)} \boxtimes \shA^{(2)} \colon p_*  \circ \eps_{k*}~.$$

Next we can apply the construction $\widetilde{P}(X_1,X_2)$ to categories $\shA^{(1)}$ and $\shA^{(2)}$. By construction of \S \ref{sec:bl} applied to $Z = \PP(\Delta_V) \subset S = \PP(V\oplus V)$, the blowing up category of $P(\shA^{(1)}, \shA^{(2)})$ along the category $P(\shA^{(1)}, \shA^{(2)})|_{\PP(\Delta_V)} \simeq \shA^{(1)} \boxtimes_{\PP(V)} \shA^{(2)}$ is
	$$\widetilde{P}(\shA^{(1)}, \shA^{(2)}) := {P}(\shA^{(1)}, \shA^{(2)})  \boxtimes_{\PP(V\oplus V)}  D(\widetilde{\PP(V\oplus V)}) \subset D(\widetilde{P}(X_1,X_2)),$$
 which is a $\PP(V)$-linear subcategory of $D(\widetilde{P}(X_1,X_2))$. The exceptional divisor $E_P(\shA^{(1)}, \shA^{(2)}) \subset E_P(X_1,X_2)$ is a projective bundle over the centre $\shA^{(1)} \boxtimes_{\PP(V)} \shA^{(2)}$ of the blow-up, and the inclusion (\ref{eq:e_P}) of exceptional divisor induces adjoint functors:
 	$$\varepsilon_{P}^* \colon \widetilde{P} (\shA^{(1)}, \shA^{(2)}) \rightleftarrows E_P(\shA^{(1)}, \shA^{(2)}) \colon \varepsilon_{P*} ~.$$
Furthermore, the strict transforms of $E_k(\shA^{(1)}, \shA^{(2)})$'s under blow-up (\ref{eq:Bl_P}) can be defined as:
 $$\widetilde{E_k}(\shA^{(1)}, \shA^{(2)}) = {E_k}(\shA^{(1)}, \shA^{(2)})  \boxtimes_{\PP(V\oplus V)}  D(\widetilde{\PP(V\oplus V)}) \subset D(\widetilde{E_k}(X_1,X_2)), \qquad k=1,2.$$
Since $E_k(X_1,X_2)$ are disjoint from the blow-up center, there are equivalences 	
	\begin{align}
	\label{eq:strict_equiv}
	p_* \circ \beta_{P*} \colon \widetilde{E_k}(\shA^{(1)}, \shA^{(2)}) \simeq {E_k}(\shA^{(1)}, \shA^{(2)}) \simeq \shA^{(1)} \boxtimes \shA^{(2)}.
	\end{align}
Notice that $\widetilde{E_k}(\shA^{(1)}, \shA^{(2)}) \subset D(\widetilde{E_k}(X_1,X_2))$ are $\PP(V)$-linear subcategories with $\PP(V)$-linear structures induced from $f|_{\widetilde{E_k}}^*$, and the induced morphisms $\eps_k^*$, $\eps_{k*}$ and also $\eps_{k!}$ are $\PP(V)$-linear.

\bigskip
As mentioned in introduction, the categorical join $\shA^{(1)} \ccup \shA^{(2)}$ will be defined by three steps. Assume from now on $\shA^{(k)}\subset D(X_k)$ are $\PP(V)$-linear Lefschetz subcategories, of length $m_k$, with Lefschetz center $\shA^{(k)}_0$ and components $\shA^{(k)}_i$, $k=1,2$. 

\medskip \noindent \textbf{Step $1$.} Following \cite{KP18}, we define the {\em categorical ruled join} $\shJ(\shA^{(1)}, \shA^{(2)})$ to be the full $\PP(V \oplus V)$-linear subcategory of $P(\shA^{(1)}, \shA^{(2)})$ characterized by:
\begin{equation*}
	{\shJ}= \shJ(\shA^{(1)}, \shA^{(2)}) := \left\{  C \in {P}(\shA^{(1)}, \shA^{(2)})  \ \left|  \
	\begin{aligned}
	{\varepsilon_1}^*(C) &\in \shA^{(1)} \boxtimes \shA^{(2)}_0 \subset {E_1}(\shA^{(1)}, \shA^{(2)}) 	 \\ 
	{\varepsilon_2}^*(C) &\in \shA^{(1)}_0 \boxtimes \shA^{(2)} \subset {E_2}(\shA^{(1)}, \shA^{(2)}) 
	\end{aligned}
	 \right.\right\}.
\end{equation*}
In the geometric case $\shA^{(k)} = D(X_k)$, $\shJ(X_1,X_2): = \shJ(D(X_1),D(X_2))$ is a categorical birational modification of the ruled join $R(X_1,X_2) \subset \PP(V\oplus V)$.

\begin{lemma} \label{lem:lef:abstrac-join}
 $\shJ$ is a $\PP(V\oplus V)$-linear Lefschetz category with respect to $\otimes \sO_{\PP(V\oplus V)}(1)$, with Lefschetz decomposition of length $m = m_1 + m_2$:
	\begin{align*} 
	\shJ = \langle \shJ_0, \shJ_1 (1), \cdots, \shJ_{m-1}((m-1))\rangle,
	\end{align*}
where the Lefschetz components are 
given by $\shJ_i : = p^* \bar{\shJ}_i$, the fully faithful image of $\bar{\shJ}_i$ under $p^*$ (where $p$ is the $\PP^1$-bundle map (\ref{eq:abstract-join})), and $\bar{\shJ}_i \subset \shA^{(1)} \boxtimes \shA^{(2)}$ is defined by:
	\begin{align*}
	& \bar{\shJ}_i : = \left \langle  \foa_{i_1}^{(1)} \otimes \foa_{i_2}^{(2)} 	\ \left|  \
	\substack{ i_1 + i_2 \ge i-1 \\
	~~\\
	i_1, i_2 \,\ge \, 0.}
	 \right. \right\rangle  \ \subset \ \shA_0^{(1)} \otimes \shA_0^{(2)} \quad \text{for} \quad i = 0, 1, \ldots, m-1,
	\end{align*}
where $\foa_{i_1}^{(k)}$'s are the primary components of the Lefschetz category $\shA^{(k)}$.
\end{lemma}
See also Thm. \ref{thm:min-rep} for general case of $n$ Lefschetz categories.
\begin{proof} This follows from the same arguments of \cite[Thm. 3.21]{KP18}. 
\end{proof}

Dually there is also a left Lefschetz decomposition for $\shJ$ with respect to $\otimes \sO_{\PP(V\oplus V)}(1)$:
	$$\shJ = \langle \shJ_{1-m}((1-m)), \cdots, \shJ_{-1} (-1), \shJ_{0} \rangle,$$
where $\shJ_i : = p^* \bar{\shJ}_{-i}$, and $\bar{\shJ}_{-i} \subset \shA^{(1)} \boxtimes \shA^{(2)}$ is defined by:
	\begin{align*}
& \bar{\shJ}_{-i} : = \left \langle  \foa_{-i_1}^{(1)} \otimes \foa_{-i_2}^{(2)} 	\ \left|  \
	\substack{ i_1 + i_2 \ge i-1 \\
	~~\\
	i_1, i_2 \,\ge \, 0.}
	 \right. \right\rangle  \ \subset \ \shA_0^{(1)} \otimes \shA_0^{(2)} \quad \text{for} \quad  i = 0, 1, \ldots, m-1.
	\end{align*}
From now on for simplicity we will simply consider right Lefschetz decompositions; the version for left Lefschetz decompositions can be proved similarly.

\medskip \noindent \textbf{Step $2$.}  Let $\widetilde{\shJ}$ be the blow-up category of $\shJ$ along the diagonal $\PP(\Delta_V) \subset \PP(V\oplus V)$,
	$$ \widetilde{\shJ}  = \widetilde{\shJ}(\shA^{(1)}, \shA^{(2)}) : =  \shJ \boxtimes_{\PP(V\oplus V)} D(\widetilde{\PP(V\oplus V)}) \subset \widetilde{P}(\shA^{(1)}, \shA^{(2)}).$$
Then $\widetilde{\shJ}$ is a $\PP(V)$-linear category, instead of a $\PP(V\oplus V)$-linear category.

 \begin{lemma} 
 \label{lem:bl:join}
 The blow-up functor $\beta_P^* \colon \shJ \to \widetilde{\shJ}$ and the functors
		$$\eps_{P *} (- \otimes \sO_{\PP(V)}(i)) \colon  \shA^{(1)} \boxtimes_{\PP(V)}  \shA^{(2)} \simeq P(\shA^{(1)}, \shA^{(2)})|_{\PP(\Delta_V)}  
		\to   \widetilde{\shJ},  \qquad i \in \ZZ$$ 
	are fully faithful, where $\beta_P$ and $\eps_{P}$ are defined in (\ref{eq:Bl_P}) and (\ref{eq:e_P}), and their images induce semiorthogonal decompositions: 
		\begin{align*}  \widetilde{\shJ} & = \langle \beta_P^* (\shJ), ~ (\shA^{(1)} \boxtimes_{\PP(V)} \shA^{(2)})_0, (\shA^{(1)} \boxtimes_{\PP(V)} \shA^{(2)})_1, \ldots, (\shA^{(1)} \boxtimes_{\PP(V)} \shA^{(2)})_{N-2} \rangle;\\
				& = \langle (\shA^{(1)} \boxtimes_{\PP(V)} \shA^{(2)})_{1-N}, \ldots, (\shA^{(1)} \boxtimes_{\PP(V)} \shA^{(2)})_{-2}, (\shA^{(1)} \boxtimes_{\PP(V)} \shA^{(2)})_{-1}, ~ \beta_P^* ( \shJ )\rangle,
			\end{align*}
	where $(\shA^{(1)} \boxtimes_{\PP(V)} \shA^{(2)})_i$ denotes the fully faithful image of $\shA^{(1)} \boxtimes_{\PP(V)} \shA^{(2)}$ under functor $\eps_{P *} (- \otimes \sO_{\PP(V)}(i))$, $i \in \ZZ$, and $N$ is the rank of $V$. 
 \end{lemma}
\begin{proof} This follows directly from blowing up formula Thm. \ref{thm:bl}.
\end{proof}

\medskip \noindent \textbf{Step $3$.}  The $\PP(V)$-linear category $\widetilde{\shJ}$ is in general not a Lefschetz category. We will remove the redundant components of $\widetilde{\shJ}$ and define the rest essential component to be the categorical join. Denote the action functor of the $\PP(V)$-linear category $\widetilde{\shJ}$ by $\act \colon \widetilde{\shJ} \boxtimes \PP(V) \to \widetilde{\shJ}$.

\begin{lemma} \label{lem:join:ref-bl} The functor $\act$ is fully faithful on the following subcategories of $\widetilde{\shJ} \boxtimes \PP(V)$:
		$$\beta_P^*(\shJ_{N-1}) \boxtimes D(\PP(V)), \ldots, \beta_P^*(\shJ_{m-1}(m-N))\boxtimes D(\PP(V)),$$
where $m=m_1 + m_2$, and their images form a semiorthogonal sequence in $\widetilde{\shJ}$. 
\end{lemma}
\begin{proof} This follows from Lem. \ref{lem:app:ref-bl} applied to the $\PP(V\oplus V)$-linear category $\shJ$ and anti-diagonal $L = -\Delta_{V^\vee} \subset V^\vee \oplus V^\vee$ (then $L^\perp = \Delta_{V} \subset V\oplus V$ and $L^\vee = V$).
\end{proof}

Denote by the category generated by images of above lemma by:
	$$\widetilde{\shJ}^{\rm amb} : = \big\langle \act\big(\beta_P^*(\shJ_{N-1}) \boxtimes D(\PP(V))\big), \ldots, \act\big( \beta_P^*(\shJ_{m-1}(m-N))\boxtimes D(\PP(V)) \big)\rangle \subset \widetilde{\shJ}.$$
 Then $\widetilde{\shJ}^{\rm amb} \subset \widetilde{\shJ}$ is a $\PP(V)$-linear subcategory. Note that $\widetilde{\shJ}^{\rm amb} = \emptyset$ if and only if $m < N$.

\begin{definition}\label{defn:cat_joins}	
For $\PP(V)$-linear Lefschetz subcategories $\shA^{(k)}$'s with Lefschetz centers $\shA^{(k)}_0$'s as above, $k=1,2$. The {\em categorical join} $\shA^{(1)} \ccup \shA^{(2)}$ of $\shA^{(1)}$ and $\shA^{(2)}$ is defined by:
	$$\shA^{(1)} \ccup \shA^{(2)} : = (\widetilde{\shJ}^{\rm amb})^\perp \subset \widetilde{\shJ}.$$
\end{definition}

From equivalences (\ref{eq:strict_equiv}), the Step $(1)$ and Step $(2)$  of above construction are {\em interchangeable}. Therefore the categorical join $\shA^{(1)} \ccup \shA^{(2)}$ can be equivalently defined as:
	\begin{equation*}
	\shA^{(1)} \ccup \shA^{(2)} := \left\{  C \in (\widetilde{\shJ}^{\rm amb})^\perp \subset \widetilde{P}(\shA^{(1)}, \shA^{(2)})  \ \left|  \
	\begin{aligned}
	\widetilde{\varepsilon_1}^*(C) &\in \shA^{(1)} \boxtimes \shA^{(2)}_0 \subset \widetilde{E_1}(\shA^{(1)}, \shA^{(2)}) 	 \\ 
	\widetilde{\varepsilon_2}^*(C) &\in \shA^{(1)}_0 \boxtimes \shA^{(2)} \subset \widetilde{E_2}(\shA^{(1)}, \shA^{(2)}) 
	\end{aligned}
	 \right.\right\}.
	\end{equation*}

\medskip
If $\shA^{(1)} \boxtimes_{\PP(V)} \shA^{(2)}= \emptyset$, then $\widetilde{P}(\shA^{(1)}, \shA^{(2)}) ={P}(\shA^{(1)}, \shA^{(2)})$ and $\widetilde{E_k}(\shA^{(1)}, \shA^{(2)}) = E_k(\shA^{(1)}, \shA^{(2)})$, $\widetilde{\eps_k}  = \eps_k$, $k=1,2$. Therefore $\shA^{(1)} \ccup \shA^{(2)} = \shJ(\shA^{(1)}, \shA^{((2)})$, and this definition agrees with the one given in \cite{KP18} which treats the case $X_k \to \PP(V_k)$, $V = V_1 \oplus V_2$.

\medskip {\em Convention}. If one of $\shA^{(k)}$ is geometric $k = 1,2$, say $\shA^{(2)} = D(S_2)$, then we use $\shA^{(1)} \ccup S_2$ to denote $\shA^{(1)} \ccup D(S_2)$. If both $\shA^{(k)}$'s are geometric, $\shA^{(k)} = D(S_k)$, then we use $S_1 \ccup S_2$ to denote the categorical join $D(S_1) \ccup D(S_2)$.

\begin{proposition}\label{prop:lef:cat_join} The categorical join $\shA^{(1)} \ccup \shA^{(2)}$ is a $\PP(V)$-linear (moderate) Lefschetz category with Lefschetz center and (right) Lefschetz components:
	$$(\shA^{(1)} \ccup \shA^{(2)})_0 = \langle \shJ_0', \sE \rangle, \qquad (\shA^{(1)} \ccup \shA^{(2)})_i  = \langle \shJ_i' , \sE \rangle, \quad 0 \le i \le N-2$$ 
where $\sE$ is the full subcategory of $\shA^{(1)} \boxtimes_{\PP(V)} \shA^{(2)}$ defined by	
	\begin{align*} 
	\sE &= \big\langle i_\Delta^* \shJ_{N}(1), \ldots, i_\Delta^* \shJ_{m-1}(m-\ell) \big \rangle^\perp \subset \shA^{(1)} \boxtimes_{\PP(V)} \shA^{(2)} \simeq \shJ|_{\PP(\Delta_V)} \\
	& = \{ C \in \shA^{(1)} \boxtimes_{\PP(V)} \shA^{(2)} \mid i_{\Delta*} C \in \langle \shJ_0(1-N), \shJ_{1}(2-N) \ldots, \shJ_{N-1} \rangle \subset \shJ\},
	\end{align*}
where $i_{\Delta} \colon \PP(\Delta_V) \hookrightarrow \PP(V \oplus V)$ denotes the inclusion, and $\shJ_i' : = \shJ_{N-1}^\perp \cap \shJ_i \subset \shJ_0$, $0 \le i \le N-2$. Note that $\shJ_i'$ can be more explicitly given by the image of
	\begin{align*}
	\left \langle \foa_{i_1}^{(1)} \otimes \foa_{i_2}^{(2)} 	\ \left|  \
	\substack{ i -1 \le i_1 + i_2  \le N-2 \\
	~~\\
	i_1, i_2 \,\ge \, 0.}
	 \right. \right\rangle  \ \subset \ \shA_0^{(1)} \boxtimes \shA_0^{(2)} \quad \text{for} \quad {i \ge 0},
	\end{align*}
under the fully faithful functor $\beta_P^* \circ p^* \colon \shA^{(1)} \boxtimes \shA^{(2)}  \to \widetilde{P}(\shA^{(1)},\shA^{(2)})$.
\end{proposition}

\begin{proof} The proof follows directly from Lem. \ref{lem:lef:abstrac-join},  Lem. \ref{lem:bl:join}, Lem. \ref{lem:join:ref-bl}, and Prop. \ref{prop:app:lef:ref-bl} applied to the $\PP(V\oplus V)$-linear category $\shJ$ and anti-diagonal $L = -\Delta_{V^\vee} \subset V^\vee \oplus V^\vee$.
\end{proof}

\begin{remark} \label{rmk:Plucker} The category $\sE$ corresponds to {\em (left) primitive components} of the intersection $\shA^{(1)} \boxtimes_{\PP(V)} \shA^{(2)}$, denoted by ${}^{\rm prim}(\shA^{(1)} \boxtimes_{\PP(V)} \shA^{(2)})$ in \cite{JLX17} (see also \cite{KP18}). The main result of \cite{JLX17} states that $\sE$ is also equivalent to the (right) primitive component of the intersection of the HPDs, i.e. 
		$${}^{\rm prim}(\shA^{(1)} \boxtimes_{\PP(V)} \shA^{(2)})\simeq (\shA^{(1,\hpd)} \boxtimes_{\PP(V^\vee)} \shA^{(2,\hpd)})^{\rm prim}.$$
	 However, we do {\em not} use this result in this paper; in fact our main Thm. \ref{thm:main} actually {\em implies} this result, hence give another proof of the nonlinear HPD theorem of \cite{JLX17, KP18}.
\end{remark}

\begin{example} \label{eg:Gr(2,5)} Let $B = \Spec \CC$,  $X_1 = \Gr(2,5) \subset \PP^9$ by Pl\"ucker embedding and $X_2 = g \cdot \Gr(2,5) \subset \PP^9$, where $g$ is a generic element of ${\rm PGL}(10,\CC)$. Then $\shA^{(1)} = D(X_1)$, $\shA^{(2)} = D(X_2)$ are $\PP^9$-linear Lefschetz categories, with natural Lefschetz structures given in \cite[\Sec 6.1]{Kuz06Hyp}. The categorical join $\shA^{(1)} \ccup \shA^{(2)}$ is the $\PP^9$-linear Lefschetz category:
	$$\Gr(2,5) \ccup (g \cdot \Gr(2,5)) = \big \langle \sE, \sE(H), \ldots, \sE(8H)\big \rangle,$$
	where $\sE = D(\Gr(2,5) \cap g \cdot \Gr(2,5) )$ is the derived category of the intersection $\Gr(2,5) \cap g \cdot \Gr(2,5)$, which is a smooth Calabi-Yau threefold for a generic $g$.
\end{example}

\begin{example} Let $\shA^{(1)} = \shA$, $\shA^{(2)} = D(\PP(L^\perp))$, where $\shA$ is any $\PP(V)$-linear Lefschetz category with Lefschetz components $\shA_k$ and length $m$, and  $L\subset V^\vee$ is a subbundle of rank $\ell$. 
Then the categorical join $\shA \ccup \PP(L^\perp)$ is the $\PP(V)$-linear Lefschetz category:
\begin{align*}
	\shA \ccup \PP(L^\perp)& =  \Big \langle   \langle \shA_0', \shC_L \rangle, \langle \shA_0', \shC_L\rangle (H), \cdots,  \langle \shA_0', \shC_L\rangle ((N-\ell)H), \\
	& \langle \shA_1', \shC_L \rangle ((N-\ell+1)H), \langle \shA_2', \shC_L \rangle ((N-\ell+2)H), \cdots, \langle \shA_{\ell-2}', \shC_L  \rangle ((N- 2)H) \Big\rangle,
	\end{align*} 
where $\shA_k' = \shA_{\ell-1}^\perp \cap \shA_k \subset \shA_0$ for $k \ge 0$, and $\shC_L$ is the same as Prop. \ref{prop:app:lef:ref-bl}, i.e.
	\begin{align*} 
	\shC_L := \{ C \in \shA_{\PP(L^\perp)} \mid i_{L*} C \in \langle \shA_0(1-\ell), \shA_{1}(2-\ell) \ldots, \shA_{\ell-1} \rangle \subset \shA\},
	\end{align*}
where $i_{L*} \colon \shA_{\PP(L^\perp)} \to \shA$ is the push-forward given by inclusion. 
In particular if $\shA_{\PP(L^\perp)} = \emptyset$, then $\shC_L = \emptyset$, $\shA_k' = \shA_k$, $\shA \ccup \PP(L^\perp)$ is a Lefschetz category of length $m+ N -\ell$, with Lefschetz components $(\shA \ccup \PP(L^\perp))_k = \shA_0$ if $0 \le k \le N-\ell$ and $(\shA \ccup \PP(L^\perp))_k = \shA_{k-(N-\ell)}$ if $N-\ell \le k \le m+N-\ell$;
If $\shA_{\PP(L^\perp)} \neq \emptyset$ but $m < \ell$, then $\shC_L = \shA_{\PP(L^\perp)}$, $\shA_k' = \shA_k$, therefore $\shA \ccup \PP(L^\perp)$ is of length $N-1$, with Lefschetz components $(\shA \ccup \PP(L^\perp))_k = \langle \shA_{0}, \shA_{\PP(L^\perp)} \rangle$ if $0 \le k \le N-\ell$, $(\shA \ccup \PP(L^\perp))_k = \langle \shA_{k-(N-\ell)}, \shA_{\PP(L^\perp)} \rangle$ for $N -\ell \le k \le N-2$. 
\end{example}

\newpage
\section{Main results} \label{sec:main-results}

\subsection{HPD between joins and intersections} 

We assume $\shA^{(k)} \subset D(X_k)$'s are Lefschetz categories of length $m_k$, $k=1,2$, where $X_k \to \PP(V)$ are smooth varieties such that $X_1 \times_{\PP(V)} X_2$ is of expected dimension. We denote $\shA^{(k), \hpd}$ the HPD category of $\shA^{(k)}$, and denote $\shC$ the {\em double HPD category} of $\shA^{(1)}$ and $\shA^{(2)}$, i.e. the $n$-HPD category Def. \ref{def:nHPD} in the case of $n=2$. Our main result of this paper is ``HPD interchanges categorical joins and intersections", i.e.

\begin{theorem}[Main theorem]\label{thm:main} There are $\PP(V^\vee)$-linear equivalences of Lefschetz categories:
	$$(\shA^{(1)} \ccup \shA^{(2)} )^\hpd \simeq \shC \simeq  \shA^{(1), \hpd} \boxtimes_{\PP(V^\vee)}  \shA^{(2), \hpd}.$$
\end{theorem}

\begin{remark} This theorem together with Lem. \ref{lem:hpdcat} implies that the  fiber products $\shA^{(1)} \boxtimes_{\PP(V)} \shA^{(2)}$ and $\shA^{(1), \hpd} \boxtimes_{\PP(V^\vee)}  \shA^{(2), \hpd}$ have a common nontrivial component, denoted by $\sE$ in Prop. \ref{prop:lef:cat_join}; see also Rmk. \ref{rmk:Plucker}. Our main theorem hence provides another proof of the nonlinear HPD theorem of \cite{JLX17, KP18}. 
\end{remark}

\begin{remark} Notice that we do not require smoothness in the definition of Lefschetz categories. In fact our theorem implies the following are equivalent: (i) $\shA^{(1)} \ccup \shA^{(2)}$ is smooth, (ii) $\shA^{(1)} \boxtimes_{\PP(V)} \shA^{(2)}$ is smooth, (iii) $\shA^{(1), \hpd} \boxtimes_{\PP(V^\vee)}  \shA^{(2), \hpd}$ is smooth, (iv) $\shC$ is smooth.
\end{remark}

\begin{example}\label{eg:Gr-dual} In the situation of Example \ref{eg:Gr(2,5)} where  $X_1 = \Gr(2,5)\subset \PP^9$, $X_2 = g \cdot \Gr(2,5)$, and $g\in {\rm PGL}(10,\CC)$ is generic, our main theorem implies that the categorical join
	$$X_1 \ccup X_2 = \Gr(2,5) \ccup (g \cdot \Gr(2,5)) = \big \langle \sE, \sE(H), \ldots, \sE(8H)\big \rangle,$$
where $\sE = D(X_1 \cap X_2)$, is HPD to the intersection $X_1^\vee \cap X_2^\vee$ (which is another Calabi-Yau threefold non-birational to $X_1 \cap X_2$, see \cite{OR, BCP}). By Lem. \ref{lem:hpdcat}, this implies that $D(X_1 \cap X_2) \simeq D(X_1^\vee \cap X_2^\vee)$. Therefore our main theorem provides another proof (cf. \cite{JLX17, KP18}) of the fact that the two Calabi-Yau threefolds are derived equivalent.
\end{example}

\begin{proof}
As mentioned in the introduction, there are two geometric pictures which allow us the prove the desired equivalences.

\medskip\noindent\emph{Step $1$.}  The generalized universal hyperplane $\shH_{P(X_1,X_2)}$ of the abstract join $P(X_1,X_2) \to \PP(V\oplus V) \dasharrow \PP(V)$ and the double universal hyerplane $\shH(X_1,X_2)$ are connected by the geometry described in \S \ref{sec:HPDbs:hyp} (which is a special case of  \S \ref{sec:hyp}).

	 Concretely, by \S \ref{sec:nHPD} the double universal hyperplane $\shH(X_1,X_2)$ is the zero locus of a canonical section $\sigma$ of the rank $2$ vector bundle $E= \sO(H_1+H') \oplus \sO(H_2 + H')$ on $X_1 \times X_2 \times \PP(V^\vee)$, 
		$$\delta_{\shH} \colon \shH(X_1,X_2) = Z(\sigma) \hookrightarrow X_1 \times X_2 \times \PP(V^\vee),$$
	where $H'$ is the hyperplane class of $\PP(V^\vee)$ and $H_k$ is the hyperplane class of $X_k \to \PP(V)$ as usual, $k=1,2$. On the other hand, under the identification
	$$H^0 (X_1 \times X_2 \times \PP(V^\vee), \sO(H_1+ H')\oplus \sO(H_2+H')) = H^0(\PP(E), \sO_{\PP(E)}(1) \otimes \sO(H')),$$
	the section $\sigma$ corresponds canonically to a section $\tilde{\sigma}$ of the line bundle $\sO_{P(X_1,X_2)}(1) \boxtimes \sO(H')$ on $\PP(\sE) = P(X_1,X_2) \times \PP(V^\vee)$. The zero locus of $\tilde{\sigma}$ is by definition the {\em generalized universal hyperplane} for the abstract join $P(X_1,X_2) \dasharrow \PP(V)$ of  \S \ref{sec:HPDbs:hyp},
		$$\delta_{\shH_P} \colon \shH_{P(X_1,X_2)} : = Z(\tilde{\sigma}) \hookrightarrow P(X_1, X_2) \times \PP(V^\vee).$$
The situation is illustrated in the following diagram, with morphisms as indicated:
\begin{equation}
	\label{diagram:cayley}
	\begin{tikzcd}[row sep= 2.6 em, column sep = 2.6 em]
	\PP(N_{\delta_{\shH}}^\vee) \ar{d}[swap]{\rho} \ar[hook]{r}{j} & \shH_{P(X_1,X_2)}  \ar{d}{\pi}  \ar[hook]{r}{\delta_{\shH_{P}}} & P(X_1,X_2) \times \PP(V^\vee) \ar{dl}{p \,\times \,\Id_{\PP(V^\vee)}}\\
	\shH(X_1,X_2) = Z(\sigma) \ar[hook]{r}{\delta_{\shH}}         &  X_1 \times X_2 \times \PP(V^\vee), 
	\end{tikzcd}	
\end{equation}

Back to categorical situation. Apply the above constructions to $\shA^{(k)} \subset D(X_k)$, $k=1,2$ as \S \ref{sec:HPDbs:hyp}. Then by Thm. \ref{thm:hyp}, there is a semiorthogonal decomposition
	\begin{equation}
	\label{eqn:Hyp:H_P}
	\shH_{P(\shA^{(1)}, \shA^{(2)})}  = \langle  j_* \, \rho^* \, \shH(\shA^{(1)}, \shA^{(2)}),  ~~\pi^* (\shA^{(1)} \boxtimes \shA^{(2)} \boxtimes D(\PP(V^\vee))) \otimes \sO_{P(X_1,X_2)}(1) \rangle.	
	\end{equation}	
By definition Def. \ref{def:nHPD} of $n$-HPD category (see also Rmk. \ref{rmk:nHPD-char}), the double HPD category $\shC \subset \shH(\shA^{(1)}, \shA^{(2)}) \subset D(\shH(X_1,X_2))$ is characterized by: 
	\begin{equation}
	\label{eqn:char:C}
	\shC =  \{C  \mid \delta_{{\shH}*} \,C \in (\shA^{(1)}_0 \boxtimes \shA^{(2)}_0) \boxtimes D(\PP(V^\vee))\} \subset \shH(\shA^{(1)}, \shA^{(2)}).
	\end{equation}
By Thm. \ref{thm:app:HPDbs}, the HPD category $(\shA^{(1)} \ccup \shA^{(2)})^\hpd$ of the categorical join $\shA^{(1)} \ccup \shA^{(2)}$ is naturally $\PP(V^\vee)$-linear equivalent to the subcategory $\shD \subset \shH_{P(\shA^{(1)}, \shA^{(2)})} \subset D(\shH_{P(X_1,X_2)})$
	\begin{equation}
	\label{eqn:char:D}
	\shD = \{D \mid \delta_{\shH\,*} \,D \in p^*( \shA^{(1)}_{0} \boxtimes \shA^{(2)}_{0}) \boxtimes D(\PP(V^\vee))\} \subset \shH_{P(\shA^{(1)}, \shA^{(2)})}.
	\end{equation}
We claim the fully faithful functor $j_* \, \rho^*$ of (\ref{eqn:Hyp:H_P}) induces a $\PP(V^\vee)$-linear equivalence:
	$$j_* \, \rho^* \colon \shC \simeq  \shD.$$
First we show $\shD$ is right orthogonal to the second component of (\ref{eqn:Hyp:H_P}). In fact for any $A \in \shA^{(1)} \boxtimes \shA^{(2)} \boxtimes D(\PP(V^\vee))$, $D \in \shD$, denote by $\sO_{P(X_1,X_2)}(1) = \sO(1)$, we have
	\begin{align*}
	& \Hom_{\shH_{P(\shA^{(1)}, \shA^{(2)})}} (\pi^* A \otimes \sO(1), D) = \Hom_{\shA^{(1)} \boxtimes \shA^{(2)} \boxtimes D(\PP(V^\vee)} \big(A, \pi_*( D \otimes \sO (-1)) \big) \\
	= & \Hom_{\shA^{(1)} \boxtimes \shA^{(2)} \boxtimes D(\PP(V^\vee)} \big(A, (p \times \Id)_* \circ \delta_{\shH_P *} (D \otimes \sO (-1)) \big)
	\end{align*}
Since $\delta_{\shH_{P \,*}} D \in p^*(\shA^{(1)}_{0} \boxtimes \shA^{(2)}_{0})  \boxtimes D(\PP^\vee)$, we may assume it is of the form $\delta_{\shH_P\,*} D = p^* D_0 \otimes F = (p \times \Id)^* (D_0 \otimes F)$ for $D_0 \in \shA^{(1)} \boxtimes \shA^{(2)}$ and $F \in D(\PP(V^\vee))$. Therefore
	\begin{align*}
	& \Hom_{\shH_{P(\shA^{(1)}, \shA^{(2)})}} (\pi^* A \otimes \sO(1), D)  \\
	 =& \Hom_{\shA^{(1)} \boxtimes \shA^{(2)} \boxtimes D(\PP(V^\vee)}  \big(A, (p \times \Id)_* \, ((p \times \Id)^*(D_0 \otimes F) \otimes \sO(-1)) \big) \\
	=&  \Hom_{\shA^{(1)} \boxtimes \shA^{(2)} \boxtimes D(\PP(V^\vee)}  \big(A, D_0 \otimes F \otimes (p_*\sO(-1)) \big)
 = 0	\end{align*}
since $p_* \sO(-1) = 0$ as $p$ is a $\PP^{1}$-bundle. Therefore $\shD$ is orthogonal to the second component of (\ref{eqn:Hyp:H_P}), and hence $\shD \subset j_* \, \rho^* \, \shH(\shA^{(1)}, \shA^{(2)})$ holds.

Next, since the ambient square in (\ref{diagram:cayley}) is Tor-independent (in fact, it is a flat base-change), we have for any $C \in \shH(\shA^{(1)}, \shA^{(2)}) \subset D(\shH(X_1,X_2)) $,
	$$ \delta_{\shH_P *} \, j_* \, \rho^* \, C = (p \times \Id)^* \, \delta_{\shH *} \, C.$$
Compare the characterizations (\ref{eqn:char:C}) and (\ref{eqn:char:D}), it is clear that for $C \in \shH(\shA^{(1)}, \shA^{(2)})$, $C \in \shC$ if and only if $ j_* \, \rho^* \, C \in \shD$. Therefore the claim is proved, i.e. we have $\PP(V^\vee)$-linear equivalences:
		$$\shC \simeq \shD \simeq (\shA^{(1)} \ccup \shA^{(2)} )^\hpd.$$

\medskip\noindent\emph{Step $2$.}  The relation between $\shC$ and the fiber product of HPDs are reflected in the geometric fact that $\shH(X_1,X_2) $ is the fiber product of $\shH_{X_1}$ and $\shH_{X_2}$ over $\PP(V^\vee)$. Let us first consider the universal case when $X_1 = X_2 = \PP(V)$. Then $\shH_{X_k} = \shH_{\PP(V)} = Q \subset \PP(V) \times \PP(V^\vee)$ is the universal quadric, and $\shH(X_1, X_2) = \shH(\PP(V), \PP(V)) = Q_2$ is the double universal hyperplane. Then $\shH(X_1, X_2) = \shH_{X_1} \times_{\PP(V^\vee)} \shH_{X_2}$, and the fiber square:
	$$
	\begin{tikzcd}
	Q_2 = \shH(\PP(V), \PP(V)) \ar{r} \ar{d}&\shH_{X_1} \ar{d} \\
	\shH_{X_2} \ar{r}	& \PP(V^\vee),
	\end{tikzcd}
	$$
 is Tor-independent. Therefore by  Prop. \ref{prop:product}, $D(Q_2) = D(\shH_{X_1}) \boxtimes_{\PP(V^\vee)} D(\shH_{X_2})$. In general for categories $\shA^{(k)}$'s, by definition (\ref{eq:def:n-Hyp}) and Lem. \ref{lem:tensor:bc}, we have natural equivalences: 
	\begin{align*}
	\shH(\shA^{(1)},\shA^{(2)}) &= (\shA^{(1)} \boxtimes \shA^{(2)}) \boxtimes_{\PP(V) \times \PP(V)} \big( D(\shH_{X_1}) \boxtimes_{\PP(V^\vee)} D(\shH_{X_2}) \big)  \\
	& = (\shA^{(1)} \boxtimes_{\PP(V)} D(\shH_{X_1})) \boxtimes_{\PP(V^\vee)}  (\shA^{(2)} \boxtimes_{\PP(V)} D(\shH_{X_2})) 
	 =  \shH_{\shA^{(1)}} \boxtimes_{\PP(V^\vee)} \shH_{\shA^{(2)}}.
	\end{align*}
By Lem. \ref{lem:sodH}, the universal hyperplane $\shH_{\shA^{(k)}}$, admits $\PP(V^\vee)$-linear  decompositions
	$$\shH_{\shA^{(k)}} = \big \langle \shA^{(k),\hpd}, ~~\shA^{(k)}_1 (H) \otimes D(\PP(V^\vee)), \ldots, \shA^{(k)}_{m_k-1}((m_k-1)H) \otimes D(\PP(V^\vee))\big\rangle, \qquad k=1,2.$$
Therefore by Prop. \ref{prop:product}, the category $\shH(\shA^{(1)},\shA^{(2)}) = \shH_{\shA^{(1)}} \boxtimes_{\PP(V^\vee)} \shH_{\shA^{(2)}}$ is equipped with a $\PP(V^\vee)$-linear semiorthogonal decomposition induced by exterior products by components of $\shH_{\shA^{(k)}}$. 
Then it is direct to see the components from exterior products of
	$$(\shA^{(1)}_{i_1} (i_1 H) \otimes D(\PP(V^\vee))) \boxtimes_{\PP(V^\vee)}  \shA^{(2),\hpd} \quad \text{and} \quad (\shA^{(2)}_{i_2} (i_2 H) \otimes D(\PP(V^\vee))) \boxtimes_{\PP(V^\vee)}  \shA^{(1),\hpd}$$
for $i_k =1,\ldots, m_k-1$, coincide with the image of (I) and (II) of the proof of Lem. \ref{lem:sod:nH}, and
	$$(\shA^{(1)}_{i1} (i_1 H) \otimes D(\PP(V^\vee))) \boxtimes_{\PP(V^\vee)}  (\shA^{(2)}_{i_2} (i_2 H) \otimes D(\PP(V^\vee))) $$
coincide with the image of (III). Comparing the decomposition from $\shH_{\shA^{(1)}} \boxtimes_{\PP(V^\vee)} \shH_{\shA^{(2)}}$ with Lem. \ref{lem:sod:nH}, we see there is $\PP(V^\vee)$-linear equivalence $\shA^{(1), \hpd} \boxtimes_{\PP(V^\vee)} \shA^{(2), \hpd} \simeq \shC$.
\end{proof}

Now we consider $n$ Lefschetz categories $\shA^{(k)} \subset D(X_k)$ of length $m_k$, where $X_k \to \PP(V)$ are smooth varieties, $k=1,\ldots,n$, where $n \ge 2$ is a fixed integer. For simplicity we assume that for every subset $I \subset \{1, 2,\ldots, n\}$, the fiber product $X_{I}$ of all $X_{i}, i\in I$ over $\PP(V)$ is {\em smooth} of expected dimension. We inductively define the {\em categorical join} of $\shA^{(k)}$, $k=1,\ldots,n$, by:
	$$\shA^{(1)} \ccup \shA^{(2)}  \ccup \cdots  \ccup \shA^{(n)} : = \big( ((\shA^{(1)} \ccup \shA^{(2)} ) \ccup \shA^{(3)}) \ccup \cdots  \big) \ccup \shA^{(n)}.$$
	
\begin{theorem}\label{thm:join-nHPD} Under the smoothness assumption of all intersections as above, the category $\underline{{\rm Lef}}_{/\PP(V)}$ of {\em smooth} proper $\PP(V)$-linear Lefschetz categories is closed under the operation of taking categorical join:
	$$  \ccup  \colon \underline{{\rm Lef}}_{/\PP(V)} \times \underline{{\rm Lef}}_{/\PP(V)} \to \underline{{\rm Lef}}_{/\PP(V)}, \qquad (\shA^{(1)} , \shA^{(2)}) \mapsto \shA^{(1)} \ccup \shA^{(2)},$$
 and the operation of taking fiber product (in the sense of noncommutative geometry):
 	$$  \boxtimes_{\PP(V)}  \colon \underline{{\rm Lef}}_{/\PP(V)} \times \underline{{\rm Lef}}_{/\PP(V)} \to \underline{{\rm Lef}}_{/\PP(V)}, \qquad (\shA^{(1)} , \shA^{(2)}) \mapsto \shA^{(1)}  \boxtimes_{\PP(V)} \shA^{(2)}.$$
These two monoidal operations are {\em commutative} and {\em associative}. Moreover, they are {\em dual} operations under the reflexive correspondence of HPD
	$$ (-)^\hpd \colon  \underline{{\rm Lef}}_{/\PP(V)}  \to \underline{{\rm Lef}}_{/\PP(V^\vee)}, \qquad \shA \mapsto \shA^\hpd.$$
Furthermore, let $\shC$ be the $n$-HPD category of $\shA^{(k)}$, $k=1,\ldots, n$, then there are a $\PP(V^\vee)$-linear equivalences of {\em smooth} Lefschetz categories:
	$$(\shA^{(1)} \ccup \shA^{(2)}  \ccup \cdots  \ccup \shA^{(n)} )^\hpd \simeq \shC \simeq  \shA^{(1), \hpd} \boxtimes_{\PP(V^\vee)}  \shA^{(2), \hpd} \boxtimes_{\PP(V^\vee)} \cdots  \boxtimes_{\PP(V^\vee)} \shA^{(n), \hpd}.$$
\end{theorem}
\begin{proof} The commutativity of categorical joins follows from definition, and the associativity of categorical joins follows from the associativity of exterior tensor products Lem. \ref{lem:ass:tensor} and the main theorem \ref{thm:main}. It remains to show the $n$-HPD category $\shC$ is equivalent to the fiber products of HPD. This follows from Step $2$ of proof of Thm. \ref{thm:main} and induction.
\end{proof}

\begin{remark} \label{rmk:cayleyn} The arguments in the proof of Thm. \ref{thm:main} can be directly applied to prove Thm. \ref{thm:join-nHPD}. In fact, a similar commutative diagram of (\ref{diagram:cayley}) holds:
	\begin{equation}
	\label{diagram:cayleyn}
	\begin{tikzcd}[row sep = 2.6 em, column sep = 2.6 em]
	\PP(N_{\delta_{\shH}}^\vee) \ar{d}[swap]{\rho} \ar[hook]{r}{j} & \shH_{P(X_1,\ldots, X_n)}  \ar{d}{\pi}  \ar[hook]{r}{\delta_{\shH_{P}}} & P(X_1, \ldots, X_n) \times \PP(V^\vee) \ar{dl}{p \,\times \,\Id_{\PP(V^\vee)}}\\
	\shH(X_1,\ldots, X_n) \ar[hook]{r}{\delta_{\shH}}         &  X_1 \times  \cdots \times X_n \times \PP(V^\vee), 
	\end{tikzcd}	
\end{equation}
and the ambient square is Tor-independent as it is a flat base-change. One can show similarly that the functor $(j_* \, \rho^*)|_{\shC}$ induces $\PP(V^\vee)$-linear equivalence $\shC \simeq \shD$, where $\shD \simeq(\shA^{(1)} \ccup  \cdots  \ccup \shA^{(n)} )^\hpd$ is the subcategory of $\shH_{P(\shA^{(1)}, \ldots, \shA^{(n)})}$ defined as in Thm. \ref{thm:app:HPDbs}.
\end{remark}

\subsection{Categorical joins and HPDs in disjoint situation} \label{sec:disjoint-sit} If $X_1,X_2 \to \PP(V)$ are disjoint $X_1 \times_{\PP(V)} X_2 = \emptyset$, then in Prop. \ref{prop:lef:cat_join}, we see that the components $\sE = \emptyset$ and the categorical join admits a simple description. We aim to generalize the disjoint situation to general $n$, and give concrete descriptions of various constructions in this situation. 

The varieties $X_k \to \PP(V)$, $k=1,2,\ldots, n$, are said to satisfy (disjoint) {condition $(D_n)$} if
\medskip
\begin{itemize}[leftmargin = 4em]
	\item[$(D_n)$ \quad] \label{cond:Dn} {\em For any $x_k \in X_k$, $k=1,\ldots, n$, the affine representatives $\hat{x}_k \in V$ of (the image of) $x_k$ in $\PP(V)$ are linearly independent, i.e. $\rank ({\rm span}\{\hat{x}_1, \ldots, \hat{x}_n\}) = n$.}
\end{itemize}
\medskip

For $n=2$, two varieties $X_1,X_2 \to \PP(V)$ satisfies $(D_2)$ if and only if they are disjoint; For $n=3$,  $X_k \to \PP(V)$, $k=1,2,3$, satisfies $(D_3)$ $\iff$ $X_1 \join X_2$ is disjoint from $X_3$ $\iff$ $X_2 \join X_3$ is disjoint from $X_1$ $\iff$ $X_1 \join X_3$ is disjoint from $X_2$; For $n\ge 3$ the condition $(D_n)$ can be described inductively in a similar manner as in the case $n=3$. 

\subsubsection{Categorical joins} \label{sec:disjoin:joins} Let $X_k \to \PP(V)$, $k=1,2,\ldots,n$ be smooth varieties that satisfy disjoint condition $(D_n)$, and $\shA^{(k)} \subset D(X_k)$ be $\PP(V)$-linear Lefschetz categories. Denote by $H_k$ the pullback of the hyperplane class of $\PP(V)$ along $X_k \to \PP(V)$ as usual. 

Although we have proved the associativity of categorical joins $\ccup$ in last section, it is not clear from definition why associativity holds. For example, the categorical join $(\shA^{(1)} \ccup \shA^{(2)}) \ccup  \shA^{(3)}$ is by definition a subcategory of the derived category of the iterated $\PP^1$-bundle:
		\begin{equation*}
		 \label{eq:P12,3}
		P_{(12)3} := \PP_{P_{12} \times X_3}(\sO(-H_{12}) \oplus \sO(-H_3)), \quad \text{where} ~P_{12} := \PP_{X_1 \times X_2}(\sO(-H_1) \oplus \sO(-H_2)),
		\end{equation*}
	and similarly $\shA^{(1)} \ccup (\shA^{(2)} \ccup \shA^{(3)})$ corresponds to the iterated $\PP^1$-bundle:
		 \begin{equation*}
		 \label{eq:P1,23}
		P_{1(23)} := \PP_{X_1 \times \rJ_{23}}(\sO(-H_{1}) \oplus \sO(-H_{23})), \quad \text{where} ~ P_{23} := \PP_{X_2 \times X_3}(\sO(-H_2) \oplus \sO(-H_3)).
		\end{equation*}
(Here $H_{12}$ and respectively $H_{23}$ denote the relative hyperplane classes of $P_{12}$ and $P_{23}$). The two spaces $P_{(12)3}$ and $P_{1(23)}$ are only birational to each other. However they both admit birational maps to $P(X_1,X_2,X_3)$, to be defined as follows.

\smallskip
For $X_k \to \PP(V)$, $k=1,2,\ldots, n$, we define their {\em abstract join} be:
	\begin{equation}
	\label{eq:pn}
	p_n \colon P(X_1, \ldots, X_n) := \PP_{X_1 \times \cdots \times X_n} \big(\bigoplus_{k=1}^n \sO(-H_k)\big) \to  X_1 \times \cdots \times X_n,
	\end{equation}
which is a $\PP^{n-1}$-bundle over $X_1 \times \cdots \times X_n$. 
We also use the notation 
	$$P_{(((12 )3 )\cdots n)} = P_{(((12 )3 )\cdots n)}(X_1,\ldots,X_n) \to X_1 \times \cdots \times X_n$$ 
to denote the {\em iterated $\PP^1$-bundle} over $X_k$, with order indicated by the bracket of the numbers on the subscript. For example, $P_{((12)3 )4}$ denotes the iterated $\PP^1$-bundle constructed by first taking $\PP^1$-bundle $P_{12} = \PP(\sO(-H_1) \oplus \sO(-H_2))$ over $X_1 \times X_2$, then taking $\PP^1$-bundle $P_{(12)3} = \PP(\sO(-H_{12}) \oplus \sO(-H_3))$ over $P_{12} \times X_3$, finally taking the $\PP^1$-bundle $P_{((12)3 )4} = \PP(\sO(-H_{(12)3}) \oplus \sO(-H_4))$ over $P_{(12)3} \times X_4$. 
It is clear from construction that every iterated $\PP^1$-bundle as above admits a birational morphism to $P(X_1,X_2,\ldots, X_n)$.

Notice that if $\{X_k \to \PP(V)\}_{k=1,\ldots,n}$ satisfies condition $(D_n)$, there is an evaluation map:
	$$ev_P \colon P(X_1, \ldots, X_n) \to \PP(V),$$
sending a point $[p_1:p_2:\ldots: p_n] \in \PP^{n-1}_{(x_1, x_2, \ldots,x_n)}$ on the fiber over $(x_1, x_2, \ldots,x_n)$ to the corresponding point $p_1x_1 + p_2x_2 +\cdots + p_n x_n \in \PP(V)$.

We can certainly apply the abstract join construction (and also iterated $\PP^1$-bundle construction) to $\PP(V)$-linear subcategories $\shA^{(k)} \subset D(X_k)$ by \S \ref{sec:app:proj_bd} and obtain $\PP^{n-1}$-bundle category $P(\shA^{(1)}, \ldots,\shA^{(n)})$ over $\shA^{(1)} \boxtimes \cdots \boxtimes \shA^{(n)}$. If condition $(D_n)$ is satisfied, then $P(\shA^{(1)}, \ldots,\shA^{(n)})$ is a $\PP(V)$-linear category, with $\PP(V)$-linear structure induced by evaluation map $ev_P$.

Next result shows that under condition $(D_n)$, there is a canonical representative for the categorical join and an explicit description of its Lefschetz structure.

\begin{theorem}\label{thm:min-rep} Let $X_k \to \PP(V)$, $k=1,\ldots,n$, be smooth varieties satisfying condition $(D_n)$, and $\shA^{(k)} \subset D(X_k)$ be Lefschetz categories. Then the categorical join $\shA^{(1)} \ccup \shA^{(2)}  \ccup \cdots  \ccup \shA^{(n)}$ is $\PP(V)$-linear equivalent to the following $\PP(V)$-linear subcategory:
	\begin{align}
	\label{eq:Jmin}
	\shJ := \langle \shJ_0, \shJ_1(H), \ldots, \shJ_{m-1} ((m-1) H)\rangle \subset P(\shA^{(1)}, \ldots,\shA^{(n)}),
	\end{align}
 where $H$ is the hyperplane class of $\PP(V)$, $m = \sum_{k=1}^n {m_k}$, and $\shJ_i := p_n^* \bar{\shJ}_i$ is the fully faithful image of $\bar{\shJ}_i$ under $p_n^*$ (where $p_n$ is the $\PP^{n-1}$-bundle morphism (\ref{eq:pn})), and 
	\begin{align}
	& \bar{\shJ}_{i}: = \left \langle  \bigotimes_{k=1}^n \foa_{i_k}^{(k)} 
	\ \left|  \
	\substack{ i_1 + \cdots + i_n \,\ge \, i+1 - n  \\ ~~\\ i_1, \ldots, i_n \,\ge \, 0.}
	 \right. \right\rangle   \ \subset \ \bigotimes_{k=1}^n \shA_0^{(k)}  \quad  i = 0, \ldots, m-1; \label{eqn:barJ>0}
	 \end{align}
 In particular, the Lefschetz center is given by $\shJ_0 := p^* (\bigotimes_{k=1}^n \shA_0^{(k)})$.
 
Furthermore, the above representative $\shJ$ of categorical join is {\em minimal} in the sense that, for any other representative of the categorical join, say for example
 	$$(((\shA^{(1)}  \ccup \shA^{(2)} ) \ccup \shA^{(3)} ) \ccup \cdots  ) \ccup \shA^{(n)}$$
	(as a subcategory of the iterated $\PP^1$-bundle $P_{(((12 )3 )\cdots n)}$), is the fully faithful image of $\shJ$ under the pullback along the birational contraction $P_{(((12 )3 )\cdots n)} \to P(X_1,\ldots, X_n)$. 
 \end{theorem}	 

The proof will be given in appendix \S \ref{app:proof:thm:min-rep}.
	 
\begin{remark} Dually, one also has the following left Lefschetz decomposition:
 	\begin{align*}
	\shJ = \langle \shJ_{1-m}((1-m)H), \cdots, \shJ_{-1} (-H), \shJ_{0} \rangle 
\subset P(\shA^{(1)}, \ldots,\shA^{(n)}),
	\end{align*}
where the Lefschetz components are given by $\shJ_{-i} := p^* \bar{\shJ}_{-i}$, and 
	 \begin{align*}
	  &  \bar{\shJ}_{-i}: = \left \langle  \bigotimes_{k=1}^n \foa_{-i_k}^{(k)}  \ \left|  \
	\substack{ i_1 + \cdots + i_n \,\ge \, i+1 - n  \\ ~~\\ i_1, \ldots, i_n \,\ge \, 0.}
	 \right. \right\rangle   \  \subset \  \bigotimes_{k=1}^n \shA_0^{(k)} , \quad  i = 0, \ldots, m-1. \label{eqn:barJ<0}
	\end{align*}
\end{remark}

\medskip
\subsubsection{$n$-HPD category} Let $X_k \to \PP(V)$, $k=1,\ldots,n$, be smooth varieties. Then projection $\pi_{\shH} \colon \shH : = \shH(X_1,\ldots, X_n) \to X_1 \times \cdots \times X_n$ from the $n$-universal hyperplane is a {\em projective bundle} precisely when $X_k \to \PP(V)$'s satisfy disjoint condition $(D_n)$. In this case the $n$-HPD category also admits an explicit description.

\begin{theorem}\label{cor:n-HPD} Assume that $X_k \to \PP(V)$, $k=1,2,\ldots, n$ satisfy disjoint condition $(D_n)$. Let $\shA^{(k)} \subset D(X_k)$, $k=1,\ldots, n$, be Lefschetz categories with centers $\shA^{(k)}_0$, and HPD categories $\shA^{(k),\hpd}$. Denote by $m_k$ the length of $\shA^{(k)}$, and by $\ell_k$ the length of HPD $\shA^{(k),\hpd}$. Then the $n$-HPD category $\shC$ of $\shA^{(k)}$'s is a Lefschetz category of length $\ell := \sum_{k=1}^n \ell_k$, with Lefschetz center:
	$$\shC_0 := \gamma_{\shC}^* \  \pi_{\shH}^* \ ( \bigotimes_{k=1}^n \shA_0^{(k)}) ,$$
where $\gamma_\shC: \shC \hookrightarrow \shH(\shA^{(1)}, \ldots, \shA^{(n)})$ is the inclusion, and $\gamma_{\shC}^*$ is its left adjoint. Moreover, the functors $\gamma_{\shC}^* \  \pi_{\shH}^*$ and its right adjoint $\pi_{\shH *} \, \gamma_\shC$ induce mutually inverse equivalences:
	$$\gamma_{\shC}^* \circ  \pi_{\shH}^* \colon \bigotimes_{k=1}^n \shA_0^{(k)} \simeq \shC_0 \quad \text{and} \quad \pi_{\shH *} \circ \gamma_\shC \colon \shC_0 \simeq \bigotimes_{k=1}^n \shA_0^{(k)} .$$
\end{theorem}

\begin{proof} This follows directly from the combination of Lefschetz structure of categorical joins (Thm. \ref{thm:min-rep}), HPD between $\shC$ and joins (Thm. \ref{thm:join-nHPD}) and Lem. \ref{lem:hpdcat}. \end{proof}

Note that all Lefschetz components of $\shC$ can be explicitly described by combining Thm. \ref{thm:min-rep} and Thm. \ref{thm:join-nHPD}; we leave these details to readers. This theorem could also be proved by directly comparing the projective bundle structure of $\shH(\shA^{(1)}, \ldots, \shA^{(n)})$ over $\shA^{(1)} \boxtimes \ldots \boxtimes \shA^{(n)}$ with Lem. \ref{lem:sod:nH}; we satisfy ourselves here with the above proof using our main theorem.


\subsection{Categorical joins and HPDs in the splitting case.}
One particular situation when the disjoint condition $(D_n)$ is satisfied is when the supports of $X_k$ are inside pairwisely disjoint linear subspaces. More precisely,   let $X_k \to \PP(V)$ be smooth projective varieties, $k=1,\ldots, n$, then the splitting condition is to say that their maps to $\PP(V)$ factor through $X_k \to \PP(V_k) \subset \PP(V)$, and $V = \bigoplus_{k=1}^n V_k$. All results of last section hold in this case. However in the splitting case, for a given $k \in \{1,\ldots, n\}$ and a $\PP(V_k)$-Lefschetz category $\shA^{(k)}$, we can either form their HPD $(\shA^{(k)})_{/\PP(V)}^{\hpd}$ over the ambient space $\PP(V)$, or their HPD $(\shA^{(k)})_{/\PP(V_k)}^{\hpd}$ over the smaller space $\PP(V_k)$. In this section we will explore the relations between these HPDs, as well as their relations with categorical joins.

For $Y_k \to \PP(V^\vee)$, $k=1,\ldots,n$, we introduce the following ``dual" notations of (\ref{eq:pn}):
	\begin{equation} 
	\label{eq:qn}
	{q}_n \colon \check{P}(Y_1, \ldots, Y_n): = \PP_{Y_1 \times \cdots \times Y_n} ~(\bigoplus_{k=1}^n \sO(-H_k')) \to Y_1 \times \cdots \times Y_n,
	\end{equation}
where $H_k'$ is the pullback of the hyperplane class $H'$ of $\PP(V^\vee)$ along $Y_k \to \PP(V^\vee)$.

Let us now describe the third geometry mentioned in introduction, which relates the double universal hyperplane with the join of (small) universal hyperplanes. This geometry only occurs in splitting case.  Let $\shH_{X_k} \subset X_k \times \PP(V_k^\vee)$ be the small universal hyperplane for $X_k \to \PP(V_k)$, $\pi_k: \shH_{X_k} \to X_k$ the projection, and $\shH(X_1,\ldots, X_n)$ be the $n$-universal hyperplane for $X_k \to \PP(V)$, $k=1,\ldots, n$. Then there is a natural birational morphism:
	\begin{equation}
	\label{eq:beta}
	\beta \colon \check{P}(\shH_{X_1}, \ldots, \shH_{X_n}) \to \shH(X_1, \ldots, X_n),
	\end{equation}
between these smooth varieties, which contracts the sections
	$$\PP_{\shH_{X_1} \times \cdots \times \shH_{X_n}}(\sO(-H_k')) \subset \check{P}(\shH_{X_1}, \ldots, \shH_{X_n}) \qquad k=1,\ldots, n,$$
to the pairwisely disjoint subvarieties
	$$\shH_{X_k} \times \prod_{k \ne i} X_k \subset \shH(X_1,\ldots, X_n).$$
(In fact $\beta$ can be constructed as the strict transform of the birational contractions 
	$\bar{\beta}: \check{P}(\PP(V_1^\vee), \ldots, \PP(V_n^\vee)) \to \PP(V_1^\vee \oplus \cdots V_n^\vee) = \PP(V^\vee)$ 
along $\shH(X_1,\ldots, X_n) \to \PP(V^\vee)$.)  
The geometric situation is summarized in the following commutative diagram:
	\begin{equation} \label{diagram:beta}
	\begin{tikzcd}[row sep=large, column sep= huge]
		\check{P}(\shH_{X_1}, \ldots, \shH_{X_n})  \ar{r}{\beta}  \ar{d}[swap]{(q,~ pr)}  \ar{rd}{g} 	&  \shH(X_1,\ldots, X_n)   \ar[hook']{d}{\delta_{\shH}} \\
		\shH_{X_1} \times \ldots \times \shH_{X_n} \times \PP(V^\vee) \ar{r}{\pi_1\times\cdots \times \pi_n \times \Id}& X_1 \times \ldots \times X_n \times \PP(V^\vee),
	\end{tikzcd}
	\end{equation}
where $q$ is the map (\ref{eq:qn}), and $pr$ is the natural evaluation map $\check{P}(\shH_{X_1},\ldots, \shH_{X_n}) \to \PP(V^\vee)$, $\delta_{\shH}$ is the map in definition of $n$-universal hyperplane, which also appears in diagram (\ref{diagram:cayley}).

In the case $n=2$, the birational morphism 
	$$\beta: \check{P}(\shH_{X_1},\shH_{X_2}) \to \shH(X_1,X_2)$$
is the blowing up of $\shH(X_1,X_2)$ along the smooth subvarieties $\shH_{X_1} \times X_2\, \sqcup\, \shH_{X_2} \times X_1\subset \shH$, with exceptional divisors $E_{1} \sqcup E_{2}$, where $E_{k} \simeq \shH_{X_1} \times \shH_{X_2}$, $k=1,2$, are the two sections of the $\PP^1$-bundle map $q \colon \check{P}(\shH_{X_1},\shH_{X_2}) \to \shH_{X_1} \times \shH_{X_2}$. (In fact, $\beta$ is the strict transform of the blowing up morphism
	$\bar{\beta}: \check{P}(\PP(V_1^\vee),\PP(V_2^\vee)) \to \PP(V_1^\vee \oplus V_2^\vee)$ 
along $\shH(X_1,X_2) \to \PP(V_1^\vee \oplus V_2^\vee)$.)

The first result compares the fiber products of HPDs and joins of HPDs.

\begin{theorem}\label{thm:join-linear}  If $X_k \to \PP(V_k) \subset \PP(V)$, $k=1,2, \ldots, n$, $n \ge 1$ are smooth projective varieties, where $V = \bigoplus_{k=1}^n V_k$, and $\shA^{(k)} \subset D(X_k)$ are Lefschetz categories.
\begin{enumerate}
	\item[$(1)$] 
There is a natural $\PP(V^\vee)$-linear equivalence of Lefschetz categories:
	\begin{align*}	
	(\shA^{(1)})_{/ \PP(V)}^\hpd  \boxtimes_{\PP(V^\vee)} \cdots  \boxtimes_{\PP(V^\vee)} (\shA^{(n)})_{/ \PP(V)}^\hpd  \simeq (\shA^{(1)})_{/ \PP(V_1)}^\hpd   \ccup \cdots \ccup  (\shA^{(n)})_{/ \PP(V_n)}^\hpd .	\end{align*}	
	\item[$(2)$] There is a natural $\PP(V^\vee)$-linear equivalence of Lefschetz categories:
	$$ 
	(\shA^{(1)} \ccup \shA^{(2)} \cdots \ccup \shA^{(n)})^\hpd_{/\PP(V)} 
	 \simeq 
	 (\shA^{(1)})_{/ \PP(V_1)}^\hpd  \ccup  (\shA^{(2)})_{/ \PP(V_2)}^\hpd  \cdots \ccup  (\shA^{(n)})_{/ \PP(V_n)}^\hpd .
	  $$
\end{enumerate}
\end{theorem}

\begin{proof} For simplicity of notations, in this proof we use $\shA^{k, \hpd} = (\shA^{k})_{/ \PP(V_k)}^\hpd$ to denote the small HPD, $\shH_{\shA^{(k)}} = \shH_{\shA^{(k)} \, /\PP(V_k)}$ to denote the small universal hyperplane, $k=1,\ldots,n$. Denote $\shC$ the $n$-HPD category for $\shA^{(k)}$ as usual , then both the statements $(1)$ and $(2)$ will follow from the following $\PP(V^\vee)$-equivalence of Lefschetz categories:
	\begin{align*} 
	\shC \simeq  (\shA^{(1)})_{/ \PP(V_1)}^\hpd  \ccup  (\shA^{(2)})_{/ \PP(V_2)}^\hpd  \cdots \ccup  (\shA^{(n)})_{/ \PP(V_n)}^\hpd=: \check{\shJ}.
	\end{align*}
Our goal, to put shortly, is to show $\beta_*$ induces an equivalence between $\check{\shJ}$ and $\shC$, where $\beta$ is the birational map (\ref{eq:beta}). To make this precise, we need to introduce certain notations.

By definition of the categorical join, there is a natural inclusion $\check{\shJ} \subset  \check{P}(\shA^{(1), \hpd}, \ldots, \shA^{(n), \hpd})$; we will not distinguish $\check{\shJ}$ with its image in the larger category, by abuse of notations.

Note that the inclusions $\gamma_k: \shA^{k, \hpd} \hookrightarrow \shH_{\shA^{(k)}}$, in the definition of HPD, induce an inclusion 
	$$\check{P}(\otimes \gamma_k) \colon  \check{P}(\shA^{(1), \hpd}, \ldots, \shA^{(n), \hpd}) \hookrightarrow \check{P}(\shH_{\shA^{(1)}}, \ldots, \shH_{\shA^{(n)}}),$$
by $\PP^{n-1}$-bundle construction of $\check{P}$ in (\ref{eq:qn}). We continue to denote the restriction of $q^*$ by:
	$$q^* \colon \bigotimes_k \shH_{\shA^{(k)}} \to \check{P}(\shH_{\shA^{(1)}}, \ldots, \shH_{\shA^{(n)}})$$
where $q = q_n$ is the $\PP^{n-1}$-bundle (\ref{eq:qn}), and denote by 
 	$$ q^*| \colon \bigotimes_k \shA^{(k), \, \hpd} \to   \check{P}(\shA^{(1), \hpd}, \ldots, \shA^{(n), \hpd}),$$
the restriction of $q^*$ to the subcategory $ \bigotimes \shA^{(k), \, \hpd}$. Therefore one has commutativity:
	$$q^*| \circ (\gamma_1 \times \cdots \times \gamma_n)^* = \check{P}(\otimes \gamma_k)^* \circ q^*.$$

Denote by $\check{\gamma}$ the restriction of the functor $\check{P}(\otimes \gamma_k)$ to $\check{\shJ}$:
		$$\check{\gamma} := \check{P}(\otimes \gamma_k)|_{\check{\shJ}}  \colon \check{\shJ} \hookrightarrow \check{P}(\shA^{(1), \hpd}, \ldots, \shA^{(n), \hpd}) \hookrightarrow \check{P} (\shH_{\shA^{(1)}}, \ldots, \shH_{\shA^{(n)}}).$$

We show that the image of $\check{\shJ}$ under $\beta_* \circ \check{\gamma}$, is contained in $\shC$ (where $\beta$ is the birational map (\ref{eq:beta})). In fact, for any $C \in \check{\shJ}$, from construction we have $q_* \circ \check{\gamma} (C) \subset \shA^{(1),\hpd} \boxtimes \cdots \boxtimes \shA^{(n), \,\hpd}$. Note that since $\pi_{k*} (\shA^{(k), \hpd}) \subset \shA^{k}_0$ (see Lem.\ref{lem:hpdcat}),  through the composition map
		$$g = \delta_{\shH} \circ \beta =(\pi_1 \times \cdots \times \pi_n \times \Id) \circ (q, pr)$$
of diagram (\ref{diagram:beta}), we have the following equality:
	$$ \delta_{\shH *} (\beta_* \circ \check{\gamma}(C))   = g_* \circ \check{\gamma}(C) \in \bigotimes_{k=1}^n \shA^{(k)}_{0} \boxtimes D(\PP(V^\vee)).$$
Therefore $\beta_* \circ \check{\gamma}(C) \in \shC$, by definition of $n$-HPD category Def. \ref{def:nHPD}. 

For simplicity of notations, denote by $\alpha$ the map $\check{\shJ} \to \shC$ induced from the of the composition of $\beta_*$ and $\check{\gamma}$, i.e. we have a commutative diagram:
	$$
	\begin{tikzcd}
	\check{P} (\shH_{\shA^{(1)}}, \ldots, \shH_{\shA^{(n)}})  \ar{r}{\beta_*}  & \shH(\shA^{(1)}, \ldots, \shA^{(n)}) \\
	\check{\shJ} \ar[hook]{u}{\check{\gamma}} \ar{r}{\alpha: = \beta_{*} \circ \check{\gamma} } & \shC \ar[hook,swap]{u}{\gamma_\shC}.
	\end{tikzcd}
	$$
	Our goal is precisely to show $\alpha$ induces a $\PP(V^\vee)$-linear equivalence of Lefschetz categories: 
	\begin{equation}
	\label{eq:alpha:equiv}
	\alpha \colon \check{\shJ} \simeq \shC.
	\end{equation}
This will follow from a direct comparison of Lefschetz centers. Concretely, by Lem. \ref{lem:equiv-lef}, we only need to show $\alpha$ and its left adjoint $\alpha^*$ induce $\PP(V^\vee)$-linear equivalences:
	$$\alpha^* \colon \shC_0 \simeq \check{\shJ}_0 \colon \alpha,$$
where the Lefschetz center $\shC_0$ for $\shC$, by Cor. \ref{cor:n-HPD}, is given by the image of
	$$  \shA_0 : = \bigotimes \shA_{0}^{(k)} \equiv  \bigotimes \shA_{0}^{(k)} \boxtimes \sO_{\PP(V^\vee)}  \subset \shA^{(1)} \boxtimes \cdots \boxtimes \shA^{(n)} \boxtimes D(\PP(V^\vee)),$$
under the map $\gamma^* \circ \delta_{\shH}^*$ , i.e.
	$$ \shC_0 = \gamma_\shC^* \circ \delta_{\shH}^* ( \shA_0 ) \subset \shC,$$
and the Lefschetz center $\check{\shJ}_0$ for $\check{\shJ}$, by Thm. \ref{thm:min-rep} and Lem. \ref{lem:hpdcat}, is given by 
	\begin{align*}
	\check{\shJ}_0 & =  q^*| ( \bigotimes \shA_{0}^{(k), \hpd}  )  
	  = q^*| \circ (\gamma_1 \times \cdots \times \gamma_n)^* \circ (\pi_1 \times \cdots \times \pi_n)^* (\shA_0)\\
	& =\check{P}(\otimes \gamma_k)^* \circ q^* \circ (\pi_1 \times \cdots \times \pi_n)^* (\shA_0)  \\
	&=  \check{\gamma}^* \circ q^* \circ (\pi_1 \times \cdots \times \pi_n)^* (\shA_0) = \check{\gamma}^* \circ g^* (\shA_0),
 	\end{align*}
where $g$ is the composition of diagram (\ref{diagram:beta}).

\medskip \noindent{\em Step $1$.} We show the left adjoint $\alpha^*$ induces $\shC_0 \simeq \check{\shJ}_0$. From 
	$$
	\alpha^* \circ \gamma_\shC^* = \check{\gamma}^* \circ \beta^* \quad \text{and} 
	\quad
	g^* = \beta^* \circ  \delta_{\shH}*,$$
the result follows immediately from above descriptions of $\shC_0$ and $\check{\shJ}_0$.

\medskip \noindent{\em Step $2$.} We show $\alpha =  \beta_{*} \circ \check{\gamma}$ induces $\check{\shJ}_0 \simeq \shC_0$. Denote 
	$$\shK_0 : =g ^* (\shA_0) \equiv  q^* (\bigotimes  \pi_k^* \, \shA_0^{(k)}) \equiv \beta^* \circ \delta_{\shH}^* (\shA_0) \subset  \check{P} (\shH_{\shA^{(1)}}, \ldots, \shH_{\shA^{(n)}}),$$
then it is clear from $\beta_* \circ \beta^* = \Id$ that $\beta_* \colon \shK_0 \simeq \delta_{\shH}^* (\shA_0) (\simeq \shA_0)$, and therefore
	$$\gamma_{\shC}^* \circ \beta_* \colon \shK_0  \simeq \shC_0.$$
Therefore we only need to show $\check{\gamma} (\check{\shJ}_0)$ and $\shK_0$ have the same image under $\beta_*$.  
Notice since
	$$\check{\gamma} (\check{\shJ}_0) = q^* (\bigotimes \gamma_k \shA_{0}^{(k), \hpd}),$$
	 is generated by elements of the form $q^* (\otimes \gamma_k \gamma_k^* a_k)$, where $a_k \in \shA_0^{(k)}$, we only need to show:
	\begin{align} 
	\label{eqn:betakills}	
	\gamma_C^* \circ \beta_* \circ q^* \, \cone (\,\otimes \pi_k^* a_k \to \otimes \gamma_k \gamma_k^* a_k ) = 0,
	\end{align}
for all $a_k \in \shA^{(k)}_0$, $k=1,\ldots, n$, where $q^*(\otimes \pi_k^* a_k)  \in \shK_0$, $q^* (\otimes\gamma_k \gamma_k^* a_k) \in \check{\gamma}(\check{\shJ}_0) $.
This will follow from the following fact:
	\begin{equation}
	\label{eq:gcone=0}
	g_* \circ q^* \circ \cone(\,\otimes \pi^* a_k \to \otimes \gamma_k \gamma_k^* a_k) = 0, \qquad \forall a_k \in \shA^{(k)}_0.
	\end{equation}

To show (\ref{eq:gcone=0}), first notice from Lem. \ref{lem:hpdcat}, one has:
	$$\pi_{k*} \, \cone(\pi_k^* a_k \to \gamma_k \, \gamma_k^* a_k) =0.$$
Since $\otimes \pi_k^* a_k \to \otimes \gamma_k \gamma_k^* a_k$ is the composition of canonical morphisms:
	\begin{align*}& \otimes a_k \equiv a_1 \otimes a_2 \otimes \cdots \otimes a_n \to \gamma_1 \gamma_1^* a_1 \otimes a_2 \otimes \cdots \otimes a_n \to 
	 \gamma_1  \gamma_1^* a_1 \otimes \gamma_2 \gamma_2^* a_2 \otimes \cdots \otimes a_n \\
	&  \to \cdots \to \gamma_1\gamma_1^* a_1\otimes \cdots \otimes  \gamma_{n-1} \gamma_{n-1}^* a_{n-1}  \otimes a_n \to  \gamma_1\gamma_1^* a_1\otimes \cdots  \otimes \gamma_n \gamma_n^* a_n \equiv  \otimes \gamma_k \gamma_k^* a_k.
	\end{align*}
(Here we write $a_k =\pi_k^* \,a_k$ for simplicity of notations.)
The cone of each of above morphisms has a tensor factor of the form $\cone( \pi_k^* a_k \to \gamma_k \, \gamma_k^* a_k)$, $k=1,\ldots,n$, therefore becomes zero under $\pi_{k*}$. Hence as an iterated cone of above cones, the $q^* \cone (\otimes \pi_k^* a_k \to \otimes \gamma_k \gamma_k^* a_k)$ becomes zero under the map $g_* = (\pi_1 \times \cdots \times \pi_n \times \Id)_* \circ (q_*, pr_*)$ (notice $q_* \, q^* = \Id$).

\medskip
To show (\ref{eq:gcone=0}) implies (\ref{eqn:betakills}), it suffices to compose the left hand side of (\ref{eqn:betakills}) by the fully faithful functor $j_* \, \rho^*$ of diagram (\ref{diagram:cayleyn}), then we will obtain: 
	$$(j_* \, \rho^*) \gamma_C^*\, \beta_* \, q^* \,\cone(\,\otimes a_k \to \otimes \gamma_k \gamma_k^* a_k )  = (\gamma_{\shD}^* \, \delta_{\shH_P}^* \, \pi^*) \,g_* \,q^* \cone(\,\otimes a_k \to \otimes \gamma_k \gamma_k^* a_k).$$
In fact, notice since $j_* \, \rho^*: \shC \simeq \shD$ (where $\shD$ denote the image which is equivalent to the HPD of categorical join $(\shA^{(1)} \ccup \cdots \ccup \shA^{(n)})^\hpd$ as Rmk. \ref{rmk:cayleyn}), we have $j_* \, \rho^* \gamma_\shC^* \simeq \gamma_{\shD}^* \,  j_* \, \rho^*$. Then from the Tor-independence of the ambient square of diagram (\ref{diagram:cayleyn}), we have
	\begin{align*}
	j_* \, \rho^* \gamma_C^* \, \beta_* \,  =  \gamma_{\shD}^* \, j_* \, \rho^* \,\beta_*  =    \gamma_{\shD}^* \, \delta_{\shH_P}^* \, \pi^* \,\delta_{\shH *}  \, \beta_* =\gamma_{\shD}^* \, \delta_{\shH_P}^* \, \pi^* \,g_*.
	\end{align*}

Therefore (\ref{eqn:betakills}) holds. To finish Step $2$, notice since $\pi_{\shH*}: \shC_0 \simeq \shA_0$, and $\pi_{\shH*} \circ \beta_* = \prod \pi_{k*} \circ\gamma_k : \check{\shJ}_0 \simeq  \shA_0$, therefore one obtains that $\beta_* : \check{\shJ}_0 \simeq \shC_0$. 

\medskip
\noindent {\em Final step.} Notice the functors $\alpha$ and $\alpha^*$ are $\PP(V^\vee)$-linear. In fact, since $\check{\shJ}$ and $\shC$ are by definition $\PP(V^\vee)$-linear subcategories of the corresponding ambient categories, and $\beta$ is a $\PP(V^\vee)$-linear morphism, therefore $\alpha = \beta_* \circ \check{\gamma}$ and its adjoint $\alpha^*$ are all $\PP(V^\vee)$-linear. From the above two steps and Lem. \ref{lem:equiv-lef}, we have an equivalence of Lefschetz categories $\beta_{*|} \colon \check{\shJ} \simeq \shC$.
\end{proof}

The next result compares the HPDs in ambient and small projective spaces. 

\begin{theorem} \label{thm:cone} If $X_k \to \PP(V_k)$, $k=1,2$, are smooth projective varieties, $V = V_1 \oplus V_2$, and $\shA^{(k)} \subset D(X_k)$ are Lefschetz categories.
\begin{enumerate}
\item[$(1)$] The homological projective dual of $\shA^{(k)}$ in $\PP(V)$ is the categorical join of its homological projective dual in $\PP(V_{k})$ and the orthogonal linear subspace of $\PP(V_{k})$, i.e. 
	\begin{align*}
	(\shA^{(1)})^{\hpd}_{/\PP(V)}  \simeq  (\shA^{(1)})^\hpd_{/\PP(V_1)}  \ccup \PP(V_2^\vee), \qquad (\shA^{(2)})^{\hpd}_{/\PP(V)}  \simeq  (\shA^{(2)})^\hpd_{/\PP(V_2)}  \ccup \PP(V_1^\vee).
	\end{align*}
\item[$(2)$]  The homological projective dual of the categorical join of $\shA^{(k)}$ and the orthogonal linear subspace is the homological projective dual of $\shA^{(k)}$ in $\PP(V_{k})$, i.e. 
	\begin{align*}
	(\shA^{(1)}  \ccup \PP(V_2))^\hpd_{/\PP(V)} \simeq (\shA^{(1)})^\hpd_{/ \PP(V_1)}, \qquad (\shA^{(2)}  \ccup \PP(V_1))^\hpd_{/\PP(V)} \simeq (\shA^{(2)})^\hpd_{/ \PP(V_2)}.
	\end{align*}
\end{enumerate}
\end{theorem}

\begin{proof} $(1)$ and $(2)$ are dual statements under HPD, so we only need to show one of them, say $(2)$. The observation is that $(2)$ can be viewed as a degenerate case of Thm. \ref{thm:join-linear}, and the same argument can be applied. The two equivalences are symmetric; we show the first equivalence. Let $\shA^{(1)} \subset D(X_1)$ be Lefschetz category, and take $X_2 = \PP(V_2)$, $\shA^{(2)} = D(\PP(V_2))$ with Beilinson's decomposition; and denote $\shC$ be the double HPD category of $\shA^{(1)}$ and $\shA^{(2)}$ as usual. Our goal is to show that everything is naturally equivalent to $\shC$: 
	$$(\shA^{(1)}  \ccup \PP(V_2))^\hpd_{/\PP(V)}   \simeq  \big((\shA^{(1)})_{/\PP(V)}^{\hpd}\big)|_{\PP(V_1^\vee)} \simeq  (\shA^{(1)})^\hpd_{/ \PP(V_1)}  \simeq \shC,$$
where $\big((\shA^{(1)})_{/\PP(V)}^{\hpd}\big)|_{\PP(V_1^\vee)}$ is the base-change category of $(\shA^{(1)})_{/\PP(V)}^{\hpd}$ along the inclusion $\PP(V_1^\vee) \to \PP(V^\vee)$. 
First, from Thm. \ref{thm:main} applied to above $\shA^{(1)}$ and $\shA^{(2)}$, we have:
	$$ (\shA^{(1)}  \ccup \PP(V_2))^\hpd_{/\PP(V)} \simeq \shC \simeq  \big((\shA^{(1)})_{/\PP(V)}^{\hpd}\big)|_{\PP(V_1^\vee)}.$$
It remains to show the last equivalence. Note $\shC$ in this case is now described by
		$$\shC  =  \{C \in \shH(\shA^{(1)}, \shA^{(2)}) \mid \delta_{\shH*} \,C \in \shA^{(1)}_0 \otimes D(\PP(V^\vee))\} \subset  \shH(\shA^{(1)}, \shA^{(2)}),$$
since $\shA^{(2)}_0 = \langle  \sO_{\PP(V_2)} \rangle$. 
Apply the construction $P^\vee$ of (\ref{eq:qn}) to the universal hyperplanes $\shH_{X_1}  = \shH(X_1)_{/ \PP(V_1)}$ and $\shH_{X_2} = \shH(\PP(V_2))_{/ \PP(V_2)}$, then we have a blowing up morphism 
	$$\beta: \check{P}(\shH_{X_1}, \shH_{X_2}) \to \shH(X_1,X_2).$$
		
The rest is the same as the proof of Thm. \ref{thm:join-linear}. In fact, since $ \delta_{\shH *} \circ \beta_* \,C= g_* \,C = ((\pi_1 \times \pi_2) \times {\rm pr}_{\PP(V^\vee)})_* \circ q_* \, C \in \shA_0  \otimes D(\PP(V^\vee))$, for any $C \in q^* (\shA^{(1)})^\hpd_{/ \PP(V_1)}$, therefore the restriction $\alpha: =\beta_{*|}$ of $\beta_*$ sends $q^*(\shA^{(1)})^\hpd_{/ \PP(V_1)} $ to $\shC$. Then the same argument of {Step $1$} shows $\alpha^* \colon \shC_0 \simeq \shA_0^{\hpd\,(1)}$, and that of {Step $2$} shows  $\alpha \colon \shA_0^{\hpd \, (1)} \simeq \shA_0^{(1)}$. Therefore by Lem \ref{lem:equiv-lef}, $\alpha$ induces a $\PP(V^\vee)$-linear equivalence between the Lefschetz categories $q^*(\shA^{(1)})^\hpd_{/ \PP(V_1)} $ and $\shC$.
\end{proof}

\begin{remark} \label{rmk:cone} For a fixed $\shA^{(1)}/\PP(V_1) \subsetneqq \PP(V)$, it follows from the theorem that $(\shA^{(1)})^{\hpd}_{/\PP(V)}  \simeq  (\shA^{(1)})^\hpd_{/\PP(V_1)}  \ccup \PP(V_2^\vee)$ holds for all $V_2$ such that $V = V_1 \oplus V_2$. In particular this implies the right hand side category is independent of the choice of $V_2$. In fact, there is another similar construction to categorical join called {\em categorical cone} $\shC_{V_1^\perp} (X_1^\hpd)$ defined in \cite{KP18}, which serves the purpose of the ``categorical join of $(\shA^{(1)})^\hpd_{/\PP(V_1)}$ and $\PP(V_1^\perp)$'', and satisfies $\shC_{V_1^\perp} (X_1^\hpd) \simeq (\shA^{(1)})^\hpd_{/\PP(V_1)} \ccup \PP(V_2^\vee)$ for every choice of splitting $V= V_1 \oplus V_2$. 
\end{remark}

\appendix
\section{$n$-HPD category and $n$-universal hyerplane} \label{app:proof:sod:nH}

\begin{proof}[Proof of Lem.\ref{lem:sod:nH}] For $n=1$ this is Lem. \ref{lem:sodH}. We focus on $n=2$, and the general case will follow from same argument and induction. Consider the Tor-independent cartesian diagram:
	\begin{equation*}
	\begin{tikzcd} [row sep= 3.6 em, column sep = 4.6 em]
	\shH   = \shH(X_1, X_2) \ar[hook]{d}[swap]{\iota_2} \ar[hook]{r}{\iota_1}  \ar[hook]{dr}{\delta_{\shH}}& \shH_{X_1} \times X_2  \ar[hook]{d}{\delta_1 \times \Id_{X_2}} \\
	\shH_{X_2} \times X_1 \ar[hook]{r}{\delta_2 \times \Id_{X_1}}         &  X_1 \times X_2 \times \PP(V^\vee). 
	\end{tikzcd}	
	\end{equation*}
We want to show that the following functors:
	\begin{align*}
	&\iota_1^* : \quad \shA^{(1), \hpd} \otimes \shA_{1}^{(2)}(H_2), \cdots, \shA^{(1), \hpd} \otimes \shA_{m_2-1}^{(2)}((m_2-1)H_2) \to \shH(\shA^{(1)},\shA^{(2)})& \text{(I)}\\
	& \iota_2^*: \quad \shA^{(2), \hpd} \otimes \shA_{1}^{(1)}(H_1), \cdots, \shA^{(2), \hpd} \otimes \shA_{m_1-1}^{(1)}((m_1-1)H_1) \to \shH(\shA^{(1)},\shA^{(2)})  & \text{(II)} \\
	 &\delta_{\shH}^*: \quad \shA_{i}^{(1)}(i H_1)\otimes \shA_{j}^{(2)}(j H_2) \otimes D(\PP(V^\vee)) \to \shH(\shA^{(1)},\shA^{(2)}), \quad i,j \ge 1. & \text{(III)}
	\end{align*}
are fully faithful, and the image are orthogonal in the desired order of the Lemma. The fully faithfulness of $\iota_k^*$ in (I) and (II), $k=1,2$ are from the fact that 
	$$\cone(\Id \to \iota_{k*} \iota_k^*) = \otimes \sO(-H_k - H')[1],$$
and that $ \shA^{(k), \hpd} $ is $\PP(V^\vee)$-linear.
More explicitly, for example consider (I), for any $C_i \otimes A_i \in  \shA^{(1), \hpd}  \otimes \shA_{j_i}^{(2)}(j_i H_2)$, $i=1,2$ where $C_i \in  \shA^{(k), \hpd} $ and $A_i \in \shA_{j_i}^{(2)}(j_i H_2)$, $j_1 \ge j_2 \ge 1$, we have:
	\begin{align*}
	& \Hom_{\shH}(\iota_1^* (C_1 \otimes A_1), \iota_1^* (C_2 \otimes A_2)) ~= ~\Hom_{\shH}( (C_1 \otimes A_1), \iota_{1*} \, \iota_1^* (C_2 \otimes A_2)) \\
	=&   \big\{ \Hom_{\shH_{X_1} \times X_2} (C_1 \otimes A_1, C_2 \otimes A_2 (-H_2-H')) \to \Hom_{\shH_{X_1} \times X_2} (C_1 \otimes A_1, C_2 \otimes A_2) \big \} \\
	=&    \big\{ \Hom_{\shH_{X_1}}(C_1,C_2(-H')) \otimes \Hom_{X_2}(A_1, A_2(-H_2)) \to \Hom_{\shH_{X_1} \times X_2} (C_1 \otimes A_1, C_2 \otimes A_2) \big \} \\
	=&   \Hom_{\shH_{X_1} \times X_2} (C_1 \otimes A_1, C_2 \otimes A_2).
	\end{align*}

The fully faithfulness of $\delta_{\shH}^*$ in (III) are from the Koszul complex for $\shH \subset X_1 \times X_2 \times \PP(V^\vee)$:
	$$0 \to \sO(-H_1-H_2-2H') \to \sO(-H_1- H') \oplus \sO(-H_2 - H') \to \sO \to \sO_{\shH},$$	
therefore $\cone(\Id \to \delta_{\shH*} \delta_{\shH}^*)$ is an iterated cone of 
	$$\otimes \sO(-H_1-H_2-2H') \quad \text{and} \quad \bigoplus_{k=1}^2\otimes \sO(-H_k-H').$$
Therefore for all the generators $A_k \otimes B_k \otimes F_k \in \shA_{i_k}^{(1)}(i_k H_1) \otimes \shA_{j_k}^{(2)}(j_k H_1) \otimes D(\PP (V^\vee))$ of (III), $k=1,2$,  where  $A_k \in \shA_{i_k}^{(1)}(i_k H_1)$, $B_k \in \shA_{j_k}^{(2)}(j_k H_2)$, $F_k \in D(\PP (V^\vee))$, $i_1\ge i_2\ge1$, $j_1 \ge j_2 \ge 1$, the cone of the following natural morphism:
	$$ \Hom_{X_1 \times X_2 \times \PP(V^\vee)}(A_1 \otimes B_1 \otimes F_1, A_2 \otimes B_2 \otimes F_2) \to \Hom_{\shH}(\delta_{\shH}^* (A_1 \otimes B_1 \otimes F_1), \delta_{\shH}^* ( A_2 \otimes B_2 \otimes F_2))$$
induced by the unit map $\Id \to \delta_{\shH*} \delta_{\shH}^*$ is an iterated cone of 
	\begin{align*} & \Hom_{X_1 \times X_2 \times \PP(V^\vee)}(A_1 \otimes B_1 \otimes F_1, A_2(-H_1) \otimes B_2(-H_2) \otimes F_2(-2H')) \\
	= & \Hom_{X_1}(A_1, A_2(-H_1))\otimes \Hom_{X_2}(B_1, B_2(-H_2)) \otimes \Hom_{\PP (V^\vee)}(F_1, F_2(-2H')) = 0
	\end{align*}
and
	\begin{align*} & \Hom(A_1 \otimes B_1 \otimes F_1, A_2(-H_1) \otimes B_2 \otimes F_2(-H')) \\
	& \bigoplus  \Hom(A_1 \otimes B_1 \otimes F_1, A_2 \otimes B_2(-H_2) \otimes F_2(-H')) 
	= 0,
	\end{align*}
where the vanishings are similarly from K\"unneth formula and $\Hom_{X_1}(A_1, A_2(-H_1)) = 0$ for $i_1 \ge i_2 \ge 1$, $\Hom_{X_2}(B_1, B_2(-H_2)) = 0$ for $j_1 \ge j_2 \ge 1$.

The fact that the image of (III) is left orthogonal to the image of (I) is from:	
\begin{align*}
	& \Hom_\shH (\delta_{\shH}^* (A_1 \otimes B_2 \otimes F), \iota_1^* (C \otimes B_2) )   \\
	 = & \Hom_{X_1 \times X_2 \times \PP(V^\vee)} (A_1 \otimes B_1 \otimes F, (\delta_1 \times \id_{X_2})_*\, \iota_{1*} \iota_1^* (C \otimes B_2)) \\
	 = & \Hom_{X_1 \times X_2 \times \PP(V^\vee)}(A_1 \otimes B_1 \otimes F,  (\delta_{1*} \times \id) \{\sO(-H_2 - H') \to \sO\} \otimes C \otimes B_2\\
	 = & \{ \Hom_{X_1 \times X_2 \times \PP(V^\vee)}(A_1 \otimes B_1 \otimes F,  \delta_{1*} C(-H') \otimes B_2(-H_2)) \to  \\
	  & \to \Hom_{X_1 \times X_2 \times \PP(V^\vee)}(A_1 \otimes B_1 \otimes F,  \delta_{1*} C \otimes B_2  \}  = 0,
	\end{align*}	
where $\delta_{\shH}^* (A_1 \otimes B_2 \otimes F)$ and $\iota_1^* (C \otimes B_2)$ are the  generators of the categories in (III) and respectively (I), where $A_1 \in \shA^{(1)}_{i_1}(i_1 H_1)$, $B_1 \in \shA^{(2)}_{i_2}(i_2 H_2)$, $C \in  \shA^{(1), \hpd} $, $B_2 \in \shA_{j}^{(2)}(j H_2)$, $F \in D(\PP V^\vee)$, $i_1,i_2, j \ge 1$. The last equality we use the fact that $\shA^{\hpd, (1)}$ is $\PP(V^\vee)$-linear, $i_1 \ge 1$, and $\delta_{1*} ( \shA^{(1), \hpd} ) \subset \shA_0^{(1)} \otimes D(\PP V^\vee)$ from the definition of $1$-HPD. The same argument shows image of (III) is left orthogonal to (II).

To show the image of (I) and (II) are complete orthogonal to each other, it suffices to observe for $C_k \in  \shA^{(k), \hpd} $, $A_k \in \shA_{i_k}^{(k)}(i_k H_k)$, $k=1,2$, $i_k \ge 1$, that
	\begin{align*} & \Hom_{\shH} (\iota_2^* (C_2 \otimes A_1), \iota_1^* (C_1 \otimes A_2)) = \Hom_{\shH_{X_2} \times X_1}(C_2 \otimes A_1, \iota_{2*} \iota_1^*(C_1 \otimes A_2) ) \\
	 = & \Hom_{\shH_{X_2} \times X_1}(C_2 \otimes A_1, (\delta_2\times\Id_{X_1})^* \, (\delta_{1} \times \Id_{X_2})_* (C_1 \otimes A_2)) \\
	 = & \Hom_{X_1 \times X_2 \times \PP(V^\vee)}( \delta_{2 !} C_2 \otimes A_1, \delta_{1*} C_1 \otimes A_2) = 0.
	\end{align*}
The last equality is from $\delta_{1*} C \in \shA_0^{(1)} \boxtimes D(\PP V^\vee)$, and the $A_1 \in \shA_{i_1}^{(1)}(i_1 H_1)$ with $i_1 \ge 1$, therefore the $\Hom_{X_1}$-factor of the above $\Hom$-space is zero. 

The characterization of the orthogonal follows from adjunctions. Denote $\shC'$ the right orthogonal of the images of (I),(II),(III) in $\shH(\shA^{(1)},\shA^{(2)})$. For any $C \in \shC'$, $\iota_{1*} C$ is right orthogonal to $ \shA^{(1), \hpd}  \otimes \shA^{(2)}_{j}(j H_2)$ for all $j \ge 1$, and the $\delta_{\shH *} C$ is right orthogonal to $\shA^{(1)}_{i_1}(i_1(H_1)) \otimes \shA^{(2)}_{j_1}(j_1 H_2) \otimes D(\PP (V^\vee))$ for all $i_1,j_1 \ge 1$. But since $\delta_{\shH *}  C = (\delta_1 \times \Id_{X_2})_* \, \iota_{1*} C$, the latter condition implies $\iota_{1*} C$ is also right orthogonal to $\delta_1^*( \shA^{(1)}_{i_1}(i_1(H_1)) \otimes D(\PP V^\vee)) \otimes \shA^{(2)}_{j_1}(j_1 H_2)$ for all $i_1,j_1 \ge 1$. From the case $n=1$, i.e. the semiorthogonal decomposition for $\shH_{\shA^{(1)}}$:
	$$\shH_{\shA^{(1)}} = \big \langle  \shA^{(1), \hpd} , \delta_1^*( \shA^{(1)}_{1}(H_1), \cdots,  \delta_1^*( \shA^{(1)}_{m_1 - 1}( (m_1-1) H_1)\big \rangle,$$
 we have $\iota_{1*} C \in \shA^{(1), \hpd}  \otimes \shA^{(2)}_0$. Then $\delta_{\shH *}  C = (\delta_1 \times \Id_{X_2})_* \, \iota_{1*} C \in \shA_0^{(1)} \otimes \shA_0^{(2)} \otimes D(\PP (V^\vee))$.

On the other hand, for any element in the $2$-HPD category $C \in \shC$, i.e. $\delta_{\shH *}  C = (\delta_1 \times \Id_{X_2})_* \, \iota_{1*} C  = (\delta_2 \times \Id_{X_1})_* \, \iota_{2*} \in \shA_0^{(1)} \otimes \shA_0^{(2)} \otimes D(\PP (V^\vee))$, then it is obvious $C$ is right orthogonal to the image of (III). Also from the characterization of $ \shA^{(k), \hpd} $ we know $\iota_{1 *} C \in  \shA^{(1), \hpd}  \otimes \shA_0^{(2)}$ and  $\iota_{2 *} C \in  \shA^{(2), \hpd}  \otimes \shA_0^{(1)}$. This implies $C$ is right orthogonal to the image of (I) and (II). Therefore $C \in \shC'$. Hence $\shC' = \shC$.
 \end{proof}

\section{Categorical joins in Disjoint situation} \label{app:proof:thm:min-rep}
	 
\begin{proof}[Proof of Thm. \ref{thm:min-rep}] We first show the following sequence of admissible subcategories:
	$$\shJ_0, \shJ_1(H), \ldots, \shJ_{m-1} ((m-1) H),$$
 of $P(\shA^{(1)}, \ldots,\shA^{(n)})$ is semiorthogonal, where $\shJ_i: = p_n^* \bar{\shJ}_i$, and $\bar{\shJ}_i$ is defined by (\ref{eqn:barJ>0}). It suffices to show for any integer $t$ such that $1 \le t \le j_1 + \cdots +  j_n + n-1$,
	$$\Hom\big(p_n^* (\bigotimes_{k=1}^n \foa_{j_k}^{(k)}) (tH)\ ,\, p_n^* \bigotimes_{k=1}^n \foa_{i_k}^{(k)}\big) = 0.$$
As $p_n$ defined in (\ref{eq:pn}) is a $\PP^{n-1}$-bundle, we have for any $t \ge 1$,
 	$$p_{n!} \sO(tH) = \bigoplus_{\substack {t_1 + \ldots + t_n = t \\ t_1, \ldots, t_n \ge 1}} \sO(\sum_{k=1}^n t_k H_k)[n-1].$$
Therefore from adjunction of $p_{n!}$ and $p^*_n$, to show the desired semiorthogonal condition it suffices to show
	$$\Hom\big(\bigotimes_{k=1}^n \foa_{j_k}^{(k)}(t_k H_k) \ , \bigotimes_{k=1}^n \foa_{i_k}^{(k)}\big) = 0.$$
This always holds since at least for one $k$ it holds that $t_k \le j_k$, therefore above $\Hom$ space vanishes from the semiorthogonal sequence of $\shA^{(k)}$ defining $\foa_{j}^{(k)}$.

Next, we show the category $\shJ$ defined by (\ref{eq:Jmin}) is a $\PP(V)$-linear, equivalent to the categorical join, and is minimal among all representatives. 

To illustrate we show in the case of $n=3$. The general situation follows from a similar argument. Back to the situation of the beginning of \S \ref{sec:disjoin:joins}, where $P_{(12)3}$ and $P_{1(23)}$ are the two iterated $\PP^1$-bundles. Denote the $\PP^1$-bundle map by $p_{(12)3}: P_{(12)3}  \to P_{12} \times X_3$ and $p_{1(23)}: P_{1(23)}  \to X_1 \times J_{23}$. For simplicity write (\ref{eq:pn}) as:
	$$p \colon P : = P(X_1,X_2,X_3) \to X_1 \times X_2 \times X_3.$$
Then there are birational morphisms 
	$$\pi_{1} \colon P_{(12)3} \to P  \quad \text{and} \quad \pi_{2} \colon P_{1(23)} \to P$$
over $X_1 \times X_2 \times X_3$, which respectively send the divisors
	 $$\PP_{P_{12} \times X_3}(\sO(-H_3)) \subset P_{(12)3}  \quad \text{and} \quad
	\PP_{X_1 \times P_{23}}(\sO(-H_{1}) ) \subset P_{1(23)}$$
to the codimension-$2$ subschemes
	$${S}_3 = \PP_{X_1 \times X_2 \times X_3}(\sO(-H_3)) \subset P \quad \text{and} \quad {S}_1 = \PP_{X_1 \times X_2 \times X_3}(\sO(-H_1)) \subset P.$$
Therefore the blow-ups of $P_{(12)3}$ and respectively  $P_{1(23)}$ along subschemes
	$$\PP_{\PP_{X_1 \times X_2}(\sO(-H_1))}(\sO(-H_3)) \subset P_{(12)3} \quad \text{and} \quad \PP_{\PP_{X_2 \times X_3}(\sO(-H_3))}(\sO(-H_1)) \subset P_{1(23)}$$
coincide, and both are equal to $\tilde{P}: = {\rm Bl}_{S_1 \sqcup S_3} P$, the blowing up of $P$ along $S_1 \sqcup S_3$. Denote the blowing-up morphisms by
	$\pi_{(12)3} \colon \tilde{P} \to P_{(12)3}$ and $\pi_{1(23)} \colon \tilde{P} \to P_{1(23)}.$
Then we have the following commutative diagram:
	\begin{equation}\label{diagram:ass}
	\begin{tikzcd}[column sep = 4.6 em]
	& \tilde{P} \ar{ld}[swap]{\pi_{(12)3}} \ar{rd}{\pi_{1(23)}}  \ar{dd}{\pi}& \\
	P_{(12)3} \ar{rd}{\pi_1}  \ar{dd}[swap]{p_{(12)3}}& & P_{1(23)} \ar{ld}[swap]{\pi_2}  \ar{dd}{p_{1(23)}} \\
	& P \ar{d}{p} & \\
	\rJ_{12} \times X_3  \ar{r}{p_{12} \times \Id_3}	&	 X_1 \times X_2 \times X_3 	& X_1 \times \rJ_{23}   \ar{l}[swap]{\Id_1 \times p_{23}}. \\
	\end{tikzcd}
	\end{equation}
Note that above geometric constructions  ($P_{(12)3}, P_{1(23)}, P, \widetilde{P}$, etc)  can clearly be applied to $\shA^{(k)} \subset D(X_k)$, $k=1,2,3$. By definition there are natural inclusions 
	\begin{align*}
	(\shA^{(1)} \ccup \shA^{(2)}) \ccup  \shA^{(3)}	 \subset P_{(12)3}(\shA^{(1)}, \shA^{(2)}, \shA^{(3)}) \subset D(P_{(12)3}); \\
 	\shA^{(1)} \ccup (\shA^{(2)} \ccup \shA^{(3)})   \subset P_{1(23)}(\shA^{(1)}, \shA^{(2)}, \shA^{(3)}) \subset  D(P_{1(23)}); 
	\end{align*}
and 
	$$\shJ   \subset P(\shA^{(1)}, \shA^{(2)},\shA^{(3)}) \subset D(P).$$
	
We want to show that :(i) the associativity of categorical joins in this case is now explicitly given by the $\PP(V)$-linear equivalence:
	$$\pi_{(12)3}^* \circ \pi_{1,23 \,*} \colon (\shA^{(1)} \ccup \shA^{(2)}) \ccup  \shA^{(3)}  \simeq \shA^{(1)} \ccup (\shA^{(2)} \ccup \shA^{(3)}),$$	
and that (ii) the fully faithful functors $\pi_1^*$ and $\pi_2^*$ induces equivalence of categories:
	\begin{equation*}
	\label{eq:pi1,pi2}
	\pi_1^* \colon \shJ \simeq (\shA^{(1)} \ccup \shA^{(2)}) \ccup \shA^{(3)}, \quad \text{and} \quad \pi_2^* \colon \shJ \simeq \shA^{(1)}  \ccup (\shA^{(2)} \ccup \shA^{(3)}).
	\end{equation*}
In fact, apply Prop. \ref{prop:lef:cat_join} to the categorical join of $\shA^{(1)} \ccup \shA^{(2)}$ and $\shA^{(3)}$, and respectively $\shA^{(1)}$ and $\shA^{(2)} \ccup \shA^{(3)}$, we see that the two representatives $(\shA^{(1)} \ccup \shA^{(2)}) \ccup \shA^{(3)}$ and $\shA^{(1)}  \ccup (\shA^{(2)} \ccup \shA^{(3)})$ both admit Lefschetz decompositions, with Lefschetz centers
	\begin{align*}
	&((\shA^{(1)} \ccup \shA^{(2)}) \ccup  \shA^{(3)})_0 = p_{(12)3}^* (p_{12}^* (\shA_0^{(1)} \otimes \shA_0^{(2)}) \otimes \shA_0^{(3)}) \subset D(P_{(12)3}); \\
 	&(\shA^{(1)} \ccup (\shA^{(2)} \ccup \shA^{(3)}))_0 = p_{1(23)}^* ( \shA_0^{(3)} \otimes p_{23}^* (\shA_0^{(2)} \otimes \shA_0^{(3)})) \subset D(P_{1(23)}).
	\end{align*}
But by the commutative diagram, they are just the fully faithful images of 
	$$\bar{\shJ}_0 = p^* (\shA_0^{(1)} \otimes \shA_0^{(2)} \otimes \shA_0^{(3)}) \subset D(P)$$
under the $P$-linear morphisms $\pi_1^*$ and respectively $\pi_2^*$. Similarly for other Lefschetz components, i.e. the $i$-th Lefschetz components of both representatives are simply given by the image of $\bar{\shJ}_i$ under $\pi_1^*$ and respectively $\pi_2^*$. It is clear then the $P$-linear functors 
	$$\pi_{(12)3}^* \circ \pi_{1,23 \,*}, \quad \text{and its left adjoint} \quad \pi_{1(23)}^* \circ \pi_{12,3 \,!}$$
induce an equivalence of the Lefschetz categories $(\shA^{(1)} \ccup \shA^{(2)}) \ccup  \shA^{(3)}$ and $\shA^{(1)} \ccup (\shA^{(2)} \ccup \shA^{(3)})$, and that there are equivalences of categories induced by adjoint functors:
	$$
	\pi_1^* \colon \shJ \simeq (\shA^{(1)} \ccup \shA^{(2)}) \ccup \shA^{(3)} \colon \pi_{1*}~, \quad \text{and} \quad \pi_2^* \colon \shJ \simeq \shA^{(1)}  \ccup (\shA^{(2)} \ccup \shA^{(3)}) \colon \pi_{1*} ~.
	$$

Notice that the functors $\pi_{(12)3}^*,~  \pi_{1,23}^*, ~\pi_1^*, ~\pi_2^*$ and their adjoints are all $\PP(V)$-linear. Since the $\PP(V)$-linear structures on 
$(\shA^{(1)} \ccup \shA^{(2)}) \ccup  \shA^{(3)}$ and $\shA^{(1)} \ccup (\shA^{(2)} \ccup \shA^{(3)})$ are both compatible with the $\PP(V)$-linear structure on $P$, which comes from the morphism $P \to \PP(V \oplus V \oplus V) \dashrightarrow \PP(V)$ (where the last map is the linear projection from the linear subspace $\PP(\Delta_{12} \oplus \Delta_{23})$, and $V \simeq \Delta_{i,j}$ are the diagonal of the direct sum $V \oplus V$ of the $i$-th and $j$-th summand). Therefore the $\PP(V)$-linear structures are compatible under above equivalences, and they hence induce a unique $\PP(V)$-linear structure on $\shJ$. The same argument clearly also works for other combinations like $(\shA^{(1)}  \ccup \shA^{(3)}) \ccup \shA^{(2)}$. This finish the proof of theorem for $n=3$. The general $n$ case follows from a similar argument.  \end{proof}
	 
\begin{remark} It would be interesting to give an intrinsic description of $\shJ$ as a $\PP(V)$-linear subcategory of $P(\shA^{(1)}, \ldots,\shA^{(n)})$.
\end{remark}



\begin{thebibliography}{99}

\bibitem[BO]{BO}
Bondal, Aleksei and Dmitri Orlov,
{\em Semiorthogonal decomposition for algebraic varieties}, 
preprint MPIM 95/15 (1995), preprint math.AG/9506012.

\bibitem[BCP]{BCP}
Borisov, Lev A , C{\u{a}}ld{\u{a}}raru, Andrei and Perry, Alexander.
{\em Intersections of two Grassmannians in $\mathbf {P}^ 9$.} 
arXiv:1707.00534 (2017).
  
\bibitem[CT]{CT15}
Carocci, F., Turcinovic Z.. 
{\em Homological projective duality for linear systems with base locus.} 
arXiv preprint arXiv:1511.09398 (2015).

\bibitem[GKZ]{GKZ}
Gelfand, Israel M., Mikhail Kapranov, and Andrei Zelevinsky. 
{\em Discriminants, resultants, and multidimensional determinants.}  Springer Science \& Business Media, 2008.

\bibitem[H06]{Huy}
Huybrechts, Daniel. 
{\em Fourier-Mukai transforms in algebraic geometry}.
Oxford University Press on Demand, 2006.

\bibitem[JLX17]{JLX17}
Jiang, Qingyuan, Naichung Conan Leung, and Ying Xie. {\em Categorical Pl\" ucker Formula and Homological Projective Duality.} arXiv:1704.01050 (2017).

\bibitem[JL18]{JL18Bl}
Jiang, Qingyuan, and Naichung Conan Leung. {\em Blowing up linear categories, refinements, and homological projective duality with base locus.} arXiv preprint (2018).

\bibitem[KR]{KR}
Kontsevich, Maxim, and Alexander L. Rosenberg. 
{\em Noncommutative smooth spaces.} In The Gelfand mathematical seminars, 1996?1999, pp. 85--108. Birkh\"auser, Boston, MA, 2000.

\bibitem[K06]{Kuz06Hyp} 
Kuznetsov, Alexander, 
{\em Hyperplane sections and derived categories},
Izvestiya RAN: Ser. Mat. 70:3 p.~23--128 (in Russian);
translation in Izvestiya: Mathematics 70:3 p.~447--547.

\bibitem[K07]{Kuz07HPD} 
Kuznetsov, Alexander, 
{\em Homological projective duality},
Publications Mathematiques de L'IHES, {\bf 105}, n. 1 (2007), 157-220.

\bibitem[K08]{Kuz08Lef}
Kuznetsov, Alexander,
{\em Lefschetz decompositions and categorical resolutions of singularities}, 
Selecta Math., 13, no. 4 (2008): 661--696.

\bibitem[K11]{Kuz11Bas} 
Kuznetsov, Alexander,
{\em  Base change for semiorthogonal decompositions.} 
Comp.\ Math.\ 147 (2011), 852--876.

\bibitem[K14]{Kuz14SOD}
Kuznetsov, Alexander, \emph{Semiorthogonal decompositions in algebraic geometry}, Proceedings of the {I}nternational {C}ongress of {M}athematicians, {V}ol.
  {II} ({S}eoul, 2014), 2014, pp.~635--660.
  
\bibitem[KP18]{KP18} Kuznetsov, Alexander, and Alexander Perry. {\em Categorical joins.} arXiv:1804.00144 (2018).
  
\bibitem[M]{Man}
Manivel, Laurent.
{\em Double spinor Calabi-Yau varieties.} arXiv:1709.07736 (2017).

\bibitem[O92]{O92}
Orlov, Dmitri. 
{\em Projective bundles, monoidal transformations, and derived categories of coherent sheaves.} 
Izvestiya Rossiiskoi Akademii Nauk. Seriya Matematicheskaya 56, no. 4 (1992): 852--862.

\bibitem[O06]{O}
Orlov, Dmitri.
{\em Triangulated categories of singularities and equivalences between Landau-Ginzburg models.} 
Sbornik: Mathematics 197.12 (2006): 1827.00--224.

\bibitem[OR]{OR}
Ottem, John Christian and Rennemo, J{\o}rgen Vold. 
{\em A counterexample to the birational Torelli problem for Calabi-Yau 3-folds.} 
arXiv:1706.09952 (2017).


\bibitem[P18]{P18}
Perry, Alexander. 
{\em Noncommutative homological projective duality.} arXiv:1804.00132 (2018).

\bibitem[T15]{RT15HPD}
Thomas, Richard. \emph{Notes on {HPD}}, to appear in Proceedings of the AMS
  Summer Institute in Algebraic Geometry, Utah 2015. arXiv:1512.08985 (2017).


\end{thebibliography}
\end{document}